\documentclass[10pt]{article}
\usepackage{latexsym, amsmath, amssymb, amsthm}
\usepackage{enumerate, mathrsfs, bbm}
\usepackage{graphicx, tikz}
\usepackage{authblk, setspace, cite}

\usepackage[top=1 in, bottom=1 in, left=1.1 in, right=1.1 in]{geometry}
\usepackage[pdftex,linkcolor=red,citecolor=blue,backref=page]{hyperref}
\numberwithin{equation}{section}

\renewcommand{\doteq}{\mathrel{:=}}
\newcommand{\fr}{\frac}
\newcommand{\lf}{\left}
\newcommand{\ri}{\right}
\newcommand{\sgn}{\operatorname{sgn}}
\newcommand{\keywords}[1]{\par\noindent\textbf{Keywords:} #1}
\newtheorem{thm}{Theorem}[section]
\newtheorem{lem}[thm]{Lemma}

\newtheorem{col}[thm]{Corollary}

\newtheorem{rem}[thm]{Remark}
\newtheorem{cla}[thm]{Claim}
\def\be{\begin{equation}}     \def\ee{\end{equation}}
\def\ba{\begin{equation*}\begin{aligned}}
\def\ea{\end{aligned}\end{equation*}}
\def\les{\lesssim}
\def\gtr{\gtrsim}

	\title{Lower regularity well-posedness for a higher-order Schrödinger equation with cubic nonlinearities on the half-line}
	\author{Yuhao Xie\\
		School of Mathematical Sciences,\\
		University of Electronic Science and Technology of China,\\
		Chengdu, Sichuan 610054, China \\
		\texttt{xieyuhao@std.uestc.edu.cn}	\and
		Shenghao Li\footnote{Corresponding author} \\
		School of Mathematical Sciences,\\
		University of Electronic Science and Technology of China,\\
		Chengdu, Sichuan 610054, China \\
		\texttt{lish@uestc.edu.cn}}
	\date{}

\begin{document}
		\maketitle
		\begin{abstract}
			In this paper,	we continue the study of Himonas and Yan\cite{himonas2024schrodinger,himonas2026higher,himonas2024higher} on the nonlinear Schrödinger equation with a dispersion of order $2m$ and cubic nonlinearities,
		\[	iu_t+\left( -1 \right) ^{m+1}\partial _{x}^{2m}u=N_k\left( u,u,u \right), \]
			where, $m\ge1$ being an integer,
			\[N_0\left( u,u,u \right) =uuu,\quad N_1\left( u,u,u \right) =\bar{u}uu,\quad N_2\left( u,u,u \right) =\bar{u}\bar{u}u,\quad  N_3\left( u,u,u \right) =\bar{u}\bar{u}\bar{u},\]
			 We derive trilinear estimates at lower regularity and thereby prove that the cNLS-2m on the half-line is well-posed at the optimal regularity $s=-\fr{m-1}{2}$. This improves the previous result, $s>-\fr{m-1}{2}$, and  answers an open question left in \cite{himonas2026higher}. In addition, for $k=0,2$, we establish well-posedness at $s=-\fr{m-1}{2}$ as well. Moreover, for $k=3$, the nonlinearity exhibits a stronger resonance relation, which implies well-posedness for $s>-\fr{2m-1}{3}$. Our derivation of the trilinear estimates relies on varieties of Strichartz estimates, which differs from the $[k;Z]$-multiplier norm method employed in \cite{himonas2026higher}.
		\end{abstract}
		\noindent \keywords{Higher-order nonlinear Schrödinger equation, cubic nonlinearities, initial-boundary value problem, trilinear estimates, Strichartz estimates.}
	\section{Introduction}
	\hspace{1em}
	In this paper, we continue the study on the nonlinear Schrödinger equation with a dispersion of order $2m$ and cubic nonlinearities(cNLS-2m),
	\be\label{1.1}
	iu_t+\left( -1 \right) ^{m+1}\partial _{x}^{2m}u=N_k\left( u,u,u \right),  
    \ee
    where, $m\ge1$ being an integer,
    \be\label{1.2}
    N_0\left( u,u,u \right) =uuu,\quad N_1\left( u,u,u \right) =\bar{u}uu,\quad N_2\left( u,u,u \right) =\bar{u}\bar{u}u,\quad  N_3\left( u,u,u \right) =\bar{u}\bar{u}\bar{u},
    \ee
    which was started by Himonas and Yan\cite{himonas2024schrodinger,himonas2026higher}. In their work\cite{himonas2026higher,himonas2024higher,himonas2024schrodinger}, an open question(see Remark 1.1 of \cite{himonas2026higher}) was brought up, more precisely, they\cite{himonas2026higher} investigated the following initial-boundary value problem(IBVP),
    \be\label{1.3}
    \left\{ \begin{array}{ll}
    	iu_t+\left( -1 \right) ^{m+1}\partial _{x}^{2m}u=N_1\left( u,u,u \right),\ x\in\mathbb{R}^{+},\ t\in\mathbb{R}^{+},\\
    	u\left( x,0 \right) =u_0\left( x \right) ,\\
    	u\left( 0,t \right) =g_0\left( t \right) ,\ \partial_{x}^{j}u\left( 0,t \right) =g_j\left( t \right) ,\ j=1,2,...,m-1, \\
    \end{array} \right.
    \ee
     and established its well-posedness for $s>s_{m,k}$ for $k=1$, with initial data $u_0\in H^{s}(\mathbb{R}^+)$ and boundary data
     $g_j\in H^{\frac{2s+2m-1-2j}{4m}}(0,T)$, $j=0,1,\dots,m-1$, where the optimal regularity $s_{m,1}$ is given by
     \be\label{1.4}
     s_{m,1}\doteq-\frac{m-1}{2}.
     \ee
      In addition, they\cite{himonas2024schrodinger} considered different nonlinearities $N_k\left( u,u,u \right)$ and proved that the following IBVP,
      \be\label{1.5}
      \left\{ \begin{array}{ll}
      	iu_t+\partial _{x}^{2}u=N_k(u,u,u),\ x\in\mathbb{R}^{+},\ t\in\mathbb{R}^{+},\\
      	u\left( x,0 \right) =u_0\left( x \right) ,\\
      	u\left( 0,t \right) =g_0\left( t \right), \\
      \end{array} \right.
      \ee
       is wellposed for $s\ge s_{m,k}$ for $m=1$, where,
       \be\label{1.6}
       s_{1,0}=s_{1,1}=s_{1,2}=0,\ s_{1,3}=(-\frac{1}{3})^+.
       \ee 
       Although these IBVPs had been well studied by Himonas and Yan, as noted in their paper, the following two problems still remain open. 
       \begin{itemize}
       	\item \textbf{For $k=1$, and $m\ge2$, $m\in\mathbb{N}^+$, is \eqref{1.3} well-posed at $s=s_{m,k}$?}
       	\item \textbf{For $k\neq1$, and $m\ge2$, $m\in\mathbb{N}^+$, is it possible to establish low-regularity well-posedness for \eqref{1.3}, in a manner analogous to \eqref{1.4}?}
       \end{itemize}
    We aim to settle these two questions in the rest of this article. Let us first consider the pure initial-value problem(IVP) of linear cNLS-2m,
   \be\label{1.7}
   \left\{ \begin{array}{l}
   	iu_t+\left( -1 \right) ^{m+1}\partial _{x}^{2m}u=0,\ x\in \mathbb{R},\ t\in\mathbb{R},\\
   	u\left( x,0 \right) =u_0\left( x \right).\\
   \end{array} \right. 
   \ee
   The fundamental solution of \eqref{1.7} can be given by, \be\label{1.8}
   G(x,t)\doteq\int_{\mathbb{R}}{e^{ix\xi}e^{-i\xi ^{2m}t}d\xi},
   \ee then, the solution of \eqref{1.7} can be written as a semigroup, 
   \be\label{1.9}
   S\left( t \right) u_0\doteq G\left( \cdot ,t \right)\ast u_0\left( \cdot \right) =\int_{\mathbb{R}}{e^{ix\xi}e^{-i\xi ^{2m}t}\widehat{u_0}\left( \xi \right) d\xi},
   \ee where $\widehat{u_0}$ denotes the space Fourier transform of $u_0$. Inspired by the idea of Bourgain\cite{bourgain1993part1,bourgain1993part2}, Kenig-Ponce-Vega\cite{kenig1993cauchy,kenig1996bilinear}, the spatial Bourgain spaces for the NLS-2m equation is defined as,
   \be\label{1.10}
   \lVert u \rVert _{X^{s,b}}\doteq\lVert \left< \xi \right> ^s\left< \tau +\xi ^{2m} \right> ^b\widehat{u}\left( \xi ,\tau \right) \rVert _{L_{\xi ,\tau}^{2}},\quad
   \left< x \right>\doteq 1+\left| x \right|
   \ee
   $\widehat{u}$ is the the space-time Fourier transform of $u$. For technical reason, the temporal Bourgain spaces,
   \be\label{1.11}
   \lVert u \rVert _{Y^{s,b}}\doteq\lVert \left< \tau \right> ^{\frac{s}{2m}}\left< \tau +\xi ^{2m} \right> ^b\widehat{u}\left( \xi ,\tau \right) \rVert _{L_{\xi ,\tau}^{2}},
   \ee
   are needed for the IBVP as well. To state our results, we now briefly review Himona-Yan's work. As mentioned in \cite{himonas2026higher}, the well-posedness problem of IBVP for cNLS-2m(or even general dispersive equation) can be settled by three steps. First, the forced linear problem \eqref{4.10}(see Section \ref{sec 4}) can be solved by the Fokas unified transform method\cite{fokas2008unified} or boundary integral operator method\cite{bona2018nonhomogeneous}. The solution is then estimated in the associated Bourgain spaces. Finally, the key multilinear estimates need to be established. A standard contraction mapping argument then gives the well-posedness. Since technique for the linear part is quite mature, the key for the problem is to establish corresponding multilinear estimates. Himonas-Yan's results are as follows. 
   \begin{lem}\cite{himonas2026higher}\label{lem 1.1}
   	For any $s>s_{m,1}$ and integer $m\ge1$, one has the spatial Bourgain spaces estimates,   	
   	\be \label{1.12}
   	\lVert N_1\left( u_1,u_2,u_3 \right) \rVert _{X^{s,-b}}\les\prod_{j=1}^3{\lVert u_j \rVert _{X^{s,b^{\prime}}}},
   	\ee
   	where $\frac{1}{2}-\beta_1\le b^{\prime}\le b<\frac{1}{2}$, $\beta_1\doteq\min \left\{ \frac{1}{16},\frac{s-s_{m,1}}{12m} \right\}$ and the temporal Bourgain spaces estimates,
   	\be\label{1.13}
   	\lVert N_1\left( u_1,u_2,u_3 \right) \rVert _{Y^{s,-b}}\les\prod_{j=1}^3{\lVert u_j \rVert _{X^{s,b^{\prime}}}},
   	\ee
   	where $\frac{1}{2}-\beta_2\le b^{\prime}\le b<\frac{1}{2}$, $\beta_2\doteq\min \left\{ \beta_1,\frac{2m-s}{16m} \right\}$. In addition, the spatial Bourgain spaces estimates \eqref{1.12} are optimal in the sense that they fail for any $s<s_{m,1}$.
   \end{lem}
\begin{rem}\label{rem 1.2}
	One can see from Lemma \ref{lem 1.1} that, as $s\to s_m$, we have, $\beta_2\to 0$, which leads to no choice for $b$ and $b^{\prime}$. Due to this fact, they can not show the well-posedness at the endpoint case, $s=s_{m,1}$.
\end{rem}
   On the other hand, for $m=1$, namely, the classical cubic nonlinear Schrödinger equation, they have the following result.
   \begin{lem}\cite{himonas2024schrodinger}\label{lem 1.3}
   	If $m=1$, then, for any $s\ge s_{1,k}$, one has,
   	\be\label{1.14}
   	\lVert N_k\left( u_1,u_2,u_3 \right) \rVert _{X^{s,-b}}\les\prod_{j=1}^3{\lVert u_j \rVert _{X^{s,b^{\prime}}}},
   	\ee
   	where $\frac{1}{2}-\alpha_k\le b^{\prime}\le b<\frac{1}{2}$, $\alpha _k=\min \left\{ \frac{1}{16},\frac{3s+1}{16} \right\}$, $s_{1,k}$ is defined by \eqref{1.6}.
   \end{lem}
   In addition, for $m=1$ and $k=1$, Erdo\u{g}an and Tzirakis\cite{erdougan2016regularity} studied the regularity properties of \eqref{1.5} for $0<s<\frac{5}{2}$. Moreover, for $m=2$ and $k=1$, Compaan, and Tzirakis\cite{compaan2023biharmonic} proved that the biharmonic Schrödinger equation on the half line,
   $$\left\{ \begin{array}{ll}
   iu_t-\partial _{x}^{4}u=\left| u \right|^2u,\,\,		x\in \mathbb{R}^+,\,\,t\in \mathbb{R}^+ ,\\
   u\left( x,0 \right) =u_0\left( x \right) ,		\\
   u\left( 0,t \right) =g_0\left( t \right) ,\ \partial _xu\left( 0,t \right) =g_1\left( t \right) ,		\\
   \end{array} \right. $$
   is well-posed for $-\fr{1}{2}<s<\fr{9}{2}$ by establishing a trilinear estimate similar as \eqref{1.12}. However, for $m\ge 2$,  trilinear estimates at $s=s_{m,1}$ are still unknown, we treat all four cubic nonlinearities
   $N_k$, $k=0,1,2,3$, in a unified manner. For simplicity in reading for the proof, we concentrate on the most delicate case, namely $s\le0$. It should be noted, however, that \eqref{1.16} remains valid for any $s>0$ and \eqref{1.17} holds for $0<s<2m$ after minor technical adjustments. We refer the reader to \cite{himonas2026higher} for further details.\par
   \begin{thm}[\textbf{Trilinear estimates}]\label{thm 1.4}
   	Let $m\ge 1$, and defined $s_{s,k}$ as follows,
   	\be\label{1.15}
   	s_{m,0}=s_{m,1}=s_{m,2}\doteq-\frac{m-1}{2},\qquad s_{m,3}\doteq-\frac{2m-1}{3},
   	\ee
   	Let $N_k$ be the nonlinearities defined below,
   	$$N_0\left( u,u,u \right) =uuu,\quad N_1\left( u,u,u \right) =\bar{u}uu,\quad N_2\left( u,u,u \right) =\bar{u}\bar{u}u,\quad  N_3\left( u,u,u \right) =\bar{u}\bar{u}\bar{u}.$$ Then the spatial and
   	temporal Bourgain-space estimates
   	\be\label{1.16}
   	\|N_k(u_1,u_2,u_3)\|_{X^{s,-b}}\les \prod_{j=1}^{3}\|u_j\|_{X^{s,b'}},
   	\ee
   	\be\label{1.17}
   	\|N_k(u_1,u_2,u_3)\|_{Y^{s,-b}}\les \prod_{j=1}^{3}\|u_j\|_{X^{s,b'}},
   	\ee
   	hold if:
   	\begin{itemize}
   		\item[(i)] {\boldmath$s_{m,k}\le s\le0$}, for $k=0,1,2$, with $b'<b$ satisfying
   		\be\label{1.18}
   		\left\{ \begin{array}{ll}
   		\tfrac{7}{18}+\tfrac{1}{36m}<b'<\tfrac12,\quad
   		\tfrac12-\tfrac{1}{8m}<b<\tfrac12,\quad
   		b'+\tfrac{b}{3}>\tfrac{7}{12}-\tfrac{1}{12m},
   		& m\ge 2,\\
   		\tfrac38<b'<b<\tfrac12,
   		& m=1 ;
   	\end{array} \right. 
   		\ee
   		\item[(ii)] {\boldmath$s_{m,k}<s\le0$}, for $k=3$, with $b'<b$ satisfying
   		\be\label{1.19}
   		\frac{1-m-3s}{2m}<b'<b<\frac12,
   		\qquad
   		b'+\frac{b}{3}>\frac{2m+1-3s}{6m}.
   		\ee
   	\end{itemize}
   \end{thm}
In case (i), the conclusion improves Lemma \ref{lem 1.1} from $s>s_{m,1}$ to the optimal regularity $s=s_{m,1}$, and also covers the cases $k=0,2$. To the best of our knowledge, the cases $k=0,2$ had not been previously addressed. In case (ii), the result generalizes Lemma \ref{lem 1.3}. Moreover, we observe that for $m=1$, $s_{1,3}=-\frac{1}{3}$, which matches the result in Lemma \ref{lem 1.3}. 
   \begin{rem}\label{rem 1.5}
   	(i) It can be verified that the admissible sets for $b$ and $b^{\prime}$ are nonempty. In addition, if we set
   	\be\label{1.20}
   	\delta:=b-b',
   	\ee
   	then for $k=0,1,2$, we have $0<\delta<\tfrac18$ when $m=1$, and when
   	$m\ge 2$ the gap $\delta$ can be any constant satisfying
   	\be\label{1.21}
   	0<\delta<\min\Big\{\tfrac19-\tfrac{1}{36m},\
   	\tfrac{1}{12}+\tfrac{1}{12m}\Big\}.
   	\ee
   	Moreover, for $k=3$ we require
   	\[
   	0<\delta<\frac{2m-1+3s}{6m}.
   	\]
   	(ii) Indeed $\delta$ can be fixed as a constant that is valid for all $m\ge 1$, $k=0,1,2$. For $k=3$, the stronger resonance relation(see \eqref{3.33}) allows for a lower-regularity estimate. However, in order to achieve this lower regularity, we still have $\delta\to 0$ as $s\to s_{m,3}$. Thus, the endpoint case for $k=3$ is not yet attainable in our current approach. Whether estimate \eqref{1.16} holds for $k=3$ at $s=s_{m,3}$ remains an interesting open problem. 
   \end{rem}
Himonas--Yan's strategy was based on Tao's $[k;Z]$-multiplier norm method(see\cite{tao2001multilinear}); more precisely, it reduced the trilinear estimate to two bilinear estimates
(or Tao's composition property). However, the endpoint case, $s=s_{m,1}$ will present difficulties. We overcome this obstacle by invoking Strichartz estimates, following the ideas of Kenig-Ponce-Vega\cite{kenig1993cauchy}, and Huo-Jia\cite{huo2005cauchy,huo2007refined}. One important Strichartz estimate for the linear cNLS-$2m$ \eqref{1.1} is given by,
\be\label{1.22}
\big\|D_x^{\frac{\alpha}{2}(1-\frac2p)}S(t)u_0\big\|_{L^q_tL^p_x}
\lesssim\|u_0\|_{L^2_x},
\ee
where $\alpha\in(-1,m-1]$, $p\in[2,+\infty]$, and $q\in\big[\frac{4m}{\alpha+1},+\infty\big)$.
Such Strichartz estimate can be derived from the decay behavior of the fundamental
solution, $G(x,t)$(see \eqref{1.8}). For $m=1$, namely, the one-dimensional classical Schr\"odinger equation, it is
well known that its decay estimate\cite{goldberg2004dispersive},
\be\label{1.23}
\sup_{x\in\mathbb{R}}|G(x,t)|\lesssim|t|^{-\fr{1}{2}},
\ee
follows directly from the Fresnel integral. However, there is no explicit formula for
the desired integral when $m\ge2$. In order to derive a proper decay estimate for the
linear NLS-$2m$, we use the scaling argument, Van der Corput's lemma and the stationary
phase method, adapting the ideas of Huo-Jia\cite{huo2005cauchy,huo2007refined}, and Seong\cite{seong2021well}. In addition, we also make use of the local smoothing effect for the linear
NLS-$2m$
\be\label{1.24}
\big\|D_x^{m-\frac12}S(t)u_0\big\|_{L^\infty_xL^2_t}
\lesssim\|u_0\|_{L^2_x},
\ee
and a Fourier restriction type estimate,
\be\label{1.25}
\big\|D_x^{-\frac14}S(t)u_0\big\|_{L^4_xL^\infty_t}
\lesssim\|u_0\|_{L^2_x},
\ee
which are used to handle the multilinear estimates in the Bourgain spaces at
different frequencies.\par
With the trilinear estimates of Theorem \ref{thm 1.4} at hand, we can now
establish our well-posedness results by using the existing linear estimates of Himonas and Yan. Before we state the main theorem, let us introduce the restriction of the
space $X^{s,b}$, defined as follows. For any $T>0$, let,
\be\label{1.26}
X^{s,b}_{\mathbb{R}^+,T}:=X^{s,b}\big|_{\mathbb{R}^+\times(0,T)},
\ee
with the quotient norm,
\be\label{1.27}
\|u\|_{X^{s,b}_{\mathbb{R}^+,T}}
:=\inf_{w\in X^{s,b}}
\big\{\|w\|_{X^{s,b}}:\ w(x,t)=u(x,t)\ \text{on}\ \mathbb{R}^+\times(0,T)\big\},
\ee
and $Y^{s,b}_{\mathbb{R}^+,T}$ is defined similarly. Our improved well-posedness result
for the IBVP can be stated as follows.
   \begin{thm}[Well-posedness for the IBVP]\label{thm 1.5}
   	Let $m\ge 1$ be an integer, $k=0,1,2,3$, and $s_{m,k}$ defined in \eqref{1.15}. Assume that $s$ satisfies,{\boldmath
   	\be\label{1.28}
   	\begin{cases}
   	s_{m,k}\le s\le0, & k=0,1,2,\\
   	s_{m,k}<s\le0, & k=3 .
   	\end{cases}
   	\ee}
   	Then, for initial data $u_0\in H^{s}(\mathbb{R}^+)$ and boundary data
   	$g_j\in H^{\frac{2s+2m-1-2j}{4m}}(0,T)$, $j=0,1,\dots,m-1$, there exist a time $0<T_0<T$ and some $b^{\prime}$,
   	$b=b(m,s,k)\in(0,\tfrac12)$, chosen as in Theorem \ref{thm 1.4} with
   	$\delta=b-b'>0$ as in Remark \ref{rem 1.5}, such that the IBVP \eqref{1.3} has a unique
   	solution $u\in X^{s,b}_{\mathbb{R}^+,T_0}$. In addition, let,
   	\be\label{1.29}
   	R\doteq\|u_0\|_{H^{s}(\mathbb{R}^+)}
   	+\sum_{j=0}^{m-1}\|g_j\|_{H^{\frac{2s+2m-1-2j}{4m}}(0,T)},
   	\ee
   	then the norm and the lifespan satisfy,
   	\be\label{1.30}
   	\|u\|_{X^{s,b}_{\mathbb{R}^+,T_0}}\les R,
   	\qquad
   	T_0\sim\langle R\rangle^{-\fr{2}{\delta}}.
   	\ee
   	Moreover, the data-to-solution map
   	is Lipschitz continuous.
   \end{thm}
\begin{rem}\label{rem 1.7}
	We comment on the optimality of the trilinear estimates \eqref{1.16}. It is known that for $k=1$, the estimate is optimal\cite{himonas2026higher}; for $k=0,2,3$, the question remains open. In particular, for $k=3$, the endpoint case $s=s_{m,3}$ is still unresolved.
\end{rem}
The rest of this paper is organized as follows. In Section \ref{sec 2}, we prove some Strichartz estimates for the linear equation, which serve as the main tool for establishing the key trilinear estimates. Section \ref{sec 3} is devoted to the proof of the trilinear estimates. The proof for the well-posedness result on the half-line is presented in Sections \ref{sec 4}.
  \section{Preliminaries}\label{sec 2}
  \hspace{1em}
  For a given function $g=g(x,t)$, let $\widehat{g}^{\left( x \right)}$ denotes the Fourier transform of $g$ with respect to the variable $x$, and $\widehat{g}$ denotes its full two-dimensional Fourier transform. We write,
  $$\widehat{D_{x}^{s}g}\doteq\left| \xi \right|^s\widehat{g}^{\left( x \right)}\left(  t,\xi \right)  ,\quad \widehat{J_{x}^{s}g}\doteq \left< \xi \right> ^s\widehat{g}^{\left( x \right)}\left( t,\xi \right).$$
    For $\rho>0$, we introduce the Fourier multipliers $F_{\rho,+}$ and $F_{\rho,-}$, defined by
  $$\widehat{F_{\rho,+} f}(\xi, \tau) \doteq \frac{\widehat{f}(\xi, \tau)}{\langle \tau + \xi^{2m} \rangle^{\rho}}, \qquad
  \widehat{F_{\rho,-} f}(\xi, \tau) \doteq \frac{\widehat{f}(\xi, \tau)}{\langle \tau - \xi^{2m} \rangle^{\rho}}.$$
  These multipliers are introduced for notational convenience, particularly in the multilinear estimates under the $X^{s,b}$ spaces. Since $F_{\rho,+}$, $F_{\rho,-}$ satisfy the same estimates, we shall unify them as $F_{\rho}$ for simplicity in notations, in the rest of this paper. In addition, we denote $F_{\rho}f$ by $F_{\rho}$ as well. Throughout the paper, we write $A\les B$ if there exists a constant $C>0$, depending only on fixed parameters, such that $A\le CB$. We write $A\sim B$ if $A\le B$ and $B\le A$.\par Next, we recall the following well-known Van der Corput's lemma\cite{stein1993harmonic}, which will be used several times in our proof. For more details of the lemma, interested reader can refer \cite{stein1993harmonic}.
  \begin{lem}\textbf{(Van der Corput)}\label{lem 2.1}
  	Suppose $\phi$ is real-valued and smooth in $(a,b)$, and that $\left| \phi ^{\left( l \right)}\left( x \right) \right|\ge 1$ for all $x\in(a, b)$, $\psi \in C^1\left[ a,b \right] $. Then,
  	\be\label{2.1}
  	\left| \int_a^b{e^{i\lambda \phi \left( x \right)}\psi \left( x \right) dx} \right|\le c_k\left| \lambda \right|^{-\frac{1}{l}}\left[ \left| \psi \left( b \right) \right|+\int_a^b{\left| \psi ^{\prime}\left( x \right) \right|dx} \right],
  	\ee
  	holds when $l=1$, $\phi^{\prime}$ is monotonic or $l\ge2$. The bound can be given by $c_l=5\cdot 2^{l-1}-2$.
  \end{lem}
In addition to the oscillatory integral estimate above, we shall also need the following inequality, known as the Hardy–Littlewood–Sobolev inequality\cite{stein1993harmonic}. 
\begin{lem}\textbf{(Hardy–Littlewood–Sobolev)}\label{lem 2.2} For $0<\gamma <n$, $\ 1<p<q<\infty$, $\frac{1}{q}=\frac{1}{p}-\frac{n-\gamma}{n}$, one has,
	\be\label{2.2}
	\lVert f\ast\left( \left| y \right|^{-\gamma} \right) \rVert _{L^q\left( \mathbb{R}^n \right)}\le C_{p,q}\lVert f \rVert _{L^p\left( \mathbb{R}^n \right)}.
	\ee
\end{lem}
Finally, we introduce the following two types of spaces. For $1 \le p, q \le \infty$, we define the mixed-norm spaces $L_t^q L_x^p$ and $L_x^p L_t^q$ by,
$$\lVert u \rVert _{L_t^q L_x^p} \doteq 
\left( \int_{\mathbb{R}} \left( \int_{\mathbb{R}} |u(x,t)|^p \, dx \right)^{q/p} dt \right)^{1/q},$$
$$\lVert u \rVert _{L_x^p L_t^q} \doteq 
\left( \int_{\mathbb{R}} \left( \int_{\mathbb{R}} |u(x,t)|^q \, dt \right)^{p/q} dx \right)^{1/p},$$
with the usual modifications when $p=\infty$ or $q=\infty$. If $p=q$, we write $L_{x,t}^{p}=L_{x}^{p}L_{t}^{q}$ or $L_{t,x}^{p}=L_{t}^{q}L_{x}^{p}$. The following subsections are based on two different Strichartz estimates on $L_{t}^{q}L_{x}^{p}$ and $L_{x}^{p}L_{t}^{q}$.
\subsection{$L_{t}^{q}L_{x}^{p}$-type estimates}
\hspace{1em}
In this subsection, we consider some $L_{t}^{q}L_{x}^{p}$-type norms of the semigroup $S(t)$ associated with the cNLS-2m equation. The first and foremost step is to establish the decay estimates for the equation. Recall that $G(x,t)\doteq\int_{\mathbb{R}}{e^{ix\xi}e^{-i\xi ^{2m}t}d\xi}$, we have the following decay estimate.
\begin{thm}\textbf{(Decay estimate)}\label{thm 2.3}
	Let $G(x,t)$ be the fundamental solution of linear cNLS-2m, defined in \eqref{1.8}, $m\ge2$ be an integer, $-1<\alpha\le m-1$, then, for any $t\in\mathbb{R}\setminus\left\{ 0 \right\}$, one has,
	\be\label{2.3}
	\left| D_{x}^{\alpha}G\left( x,t \right) \right|\le C_{m,\alpha}\left| t \right|^{-\frac{\alpha +1}{2m}}.
	\ee 
\end{thm}
\begin{proof}
	To estimate the oscillatory integral, we first perform the scaling $\xi =\left| t \right|^{-\frac{1}{2m}}\eta$, $y =\left| t \right|^{-\frac{1}{2m}}x$, then, a direct computation gives,
	\be\label{2.4}
	D_{x}^{\alpha}G\left( x,t \right) =\left| t \right|^{-\frac{\alpha +1}{2m}}\int_{\mathbb{R}}{\left| \eta \right|^{\alpha}e^{i\left( y\eta -\sgn \left( t \right) \eta ^{2m} \right)}d\eta}\doteq \left| t \right|^{-\frac{\alpha +1}{2m}} I\left( y \right),
	\ee
	with,
	\be\label{2.5}
	I\left( y \right) =\int_0^{+\infty}{\eta^{\alpha}e^{-i\left( \sgn \left( t \right) \eta ^{2m}+y\eta  \right)}d\eta}+\int_0^{+\infty}{ \eta^{\alpha}e^{-i\left(  \sgn \left( t \right) \eta ^{2m}-y\eta \right)}d\eta}\doteq I_1\left( y \right) +I_2\left( y \right).
	\ee
	It thus remains to show $I_1\left( y \right)$ and $I_2\left( y \right)$ are bounded. We will only treat $t>0$, since the proof for $t<0$ is similar. In addition, we notice the symmetry, $I_1\left( y \right) =I_2\left( -y \right)$, thus, we only need to prove that both $I_1\left( y \right)$ and $I_2\left( y \right)$ are bounded for $y\ge0$. 
	For simplicity, we write $\varPhi _1\left( \eta \right) =\eta ^{2m}+y\eta$ and $\varPhi _2\left( \eta \right) =\eta ^{2m}-y\eta$.
	\begin{itemize}
		\item \textbf{Estimate of $I_1\left( y \right)$}\quad If {\boldmath $\eta \le 2$}, then, it can be checked that $\left| I_1\left( y \right) \right|\le 2^{\alpha +1}$. If {\boldmath $\eta \ge 2$}, then, we consider the integral with $\eta \in J_k:=\left[ 2^k,2^{k+1} \right]$, $k\in\mathbb{N}^{+}$. On every $J_k$, we have, $\varPhi _{1}^{\prime}\left( \eta \right) \ge 2m2^{\left( 2m-1 \right) k}$ since $y\ge0$. Thus, by the first-order Van der Corput's lemma combined with the condition $-1<\alpha\le m-1$, we have,
		\be\label{2.6}
		\left| \int_{I_k}{\eta ^{\alpha}e^{-i\varPhi _1\left( \eta \right)}d\eta} \right|\le \frac{3}{2m2^{\left( 2m-1 \right) k}}\left[ 2^{\alpha \left( k+1 \right)}+\left| 2^{\alpha \left( k+1 \right)}-2^{\alpha k} \right| \right] \le \frac{3\cdot 2^{\alpha}}{2m}\frac{1}{2^{mk}}.
		\ee
		Summing over all dyadic intervals $J_k$ yields,
		\be\label{2.7}
		\left| \int_2^{+\infty}{\eta ^{\alpha}e^{-i\varPhi _1\left( y \right)}d\eta} \right|\le \sum_{k=1}^{\infty}{\frac{3\cdot 2^{\alpha}}{2m}\frac{1}{2^{mk}}}= \frac{3\cdot 2^{\alpha}}{2m\left( 2^m-1 \right)},
		\ee
		which implies that $\left| I_1\left( y \right) \right|\le \max \left\{ 2^{\alpha +1},\frac{3\cdot 2^{\alpha}}{2m\left( 2^m-1 \right)}\ri\} $.
		\item \textbf{Estimate of $I_2\left( y \right)$}\quad By direct computation, one has $\varPhi _{2}^{\prime}\left( \eta \right) =2m\eta ^{2m-1}-y$, and $\varPhi _{2}^{\prime\prime}\left( \eta \right) =2m\left( 2m-1 \right) \eta ^{2m-2}$, thus $\varPhi_2$ has a unique stationary point $\eta_0:=\left({\frac{y}{2m}}\right)^{\frac{1}{2m-1}}$ which implies $y=2m{\eta_0}^{2m-1}$. We now consider the following cases.
		\subitem \textbf{For} {\boldmath $\eta_0\le1$.}\quad Since $\varPhi^{\prime}_{2}\ge 2m2^{k\left( 2m-1 \right)}-y=2m\left( 2^{k\left( 2m-1 \right)}-\eta _0^{2m-1} \right) \ge 2m\left( 2^{k\left( 2m-1 \right)}-1 \right) $ on each $J_k$, the estimate for $I_2$ follows by the same argument as that used in \eqref{2.6} and \eqref{2.7}.
		\subitem \textbf{For} {\boldmath$\eta_0>1$.}\quad Next, we further consider the following three cases. 
		\subsubitem \textbf{Case A:}\ {\boldmath$\eta\le\frac{{\eta}_0}{2}$}\quad In this case, we have $\left|\varPhi^{\prime}_{2}\left(\eta\right)\right|\ge\left|\varPhi^{\prime}_{2}\left(\frac{\eta_0}{2}\right)\right|\ge y-2m\left(\frac{\eta_0}{2}\right)^{2m-1}\ge m{\eta_0}^{2m-1}$. Then, by the first-order Van der Corput's lemma  together with $\alpha\le m-1$, $\eta_0\ge2$, we can obtain
		\be
		\left| \int_0^{\frac{\eta _0}{2}}{\eta ^{\alpha}e^{-i\varPhi _2\left( \eta \right)}d\eta} \right|\le \frac{3}{m\eta _{0}^{2m-1}}\left[ \left( \frac{\eta _0}{2} \right) ^{\alpha}+\left( \frac{\eta _0}{2} \right) ^{\alpha} \right] \le \frac{6}{m2^{\alpha}}\cdot \eta _0^{\alpha -2m+1}\le \frac{3}{m2^{\alpha}}.
		\ee
		\subsubitem \textbf{Case B:}\ {\boldmath$\frac{{\eta}_0}{2}\le\eta\le2\eta_0$}\quad One can check that $\lf|\varPhi_{2}^{\prime\prime}\lf(\eta\ri)\ri|\ge\lf|\varPhi_{2}^{\prime\prime}\lf(\frac{\eta_0}{2}\ri)\ri|=\frac{m\left(2m-1\right)}{2^{2m-3}}{\eta_0}^{2m-2}:=c_m{\eta_0}^{2m-2}$, then using the second-order Van der Corput's lemma will lead to
		\be
		\left| \int_{\frac{\eta _0}{2}}^{2\eta _0}{\eta ^{\alpha}e^{-i\varPhi _2\left( \eta \right)}d\eta} \right|\le \frac{8}{\sqrt{c_m}\eta _0^{m-1}}\left[ \left( 2\eta _0 \right) ^{\alpha}+\left| \left( 2\eta _0 \right) ^{\alpha}-\left( \frac{\eta _0}{2} \right) ^{\alpha} \right| \right] \le \frac{2^{\alpha +3}}{\sqrt{c_m}}\eta _0^{\alpha -m+1}\le \frac{2^{\alpha +3}}{\sqrt{c_m}}.
		\ee
		\subsubitem \textbf{Case C:}\ {\boldmath$\eta\ge2\eta_0$}\quad For $k\in\mathbb{N}$, we consider the integral with $\eta \in Q_k:=\left[ 2^k\cdot2\eta_0,2^{k+1}\cdot2\eta_0 \right]$. On each $Q_k$, recall that $y=2m{\eta_0}^{2m-1}$, then, monotonicity of $\varPhi^{\prime}_{2}$ gives $$\varPhi_{2}^{\prime}\left(\eta\right)\ge\varPhi_{2}^{\prime}\lf(2^k\cdot2\eta_0\ri)=2m\eta _0^{2m-1}\left( 2^{\left( k+1 \right) \left( 2m-1 \right)}-1 \right) \ge m\eta _0^{2m-1}2^{\left( k+1 \right) \left( 2m-1 \right)}\ge m2^{\left( k+1 \right) \left( 2m-1 \right)}.$$ Hence the first-order Van der Corput's lemma applies on every $Q_k$, yielding the following estimate.
		\be
		\left| \int_{2\eta _0}^{+\infty}{\eta ^{\alpha}e^{-i\varPhi _2\left( \eta \right)}d\eta} \right|\le \sum_{k=0}^{\infty}{\left| \int_{Q_k}{\eta ^{\alpha}e^{-i\varPhi _2\left( \eta \right)}d\eta} \right|}\le \sum_{k=0}^{\infty}{\frac{c_{m,\alpha}}{2^{\left( k+1 \right)\left( 2m-1 \right)}}}=\frac{c_{m,\alpha}}{2^{2m-1}-1},
		\ee
		where $c_{m,\alpha}$ is a positive constant depending only on $m$ and $\alpha$. 
	\end{itemize}
    Therefore, taking the maximum of the above bounds gives a uniform estimate, and the proof is complete.\end{proof}
The Strichartz estimates of $L^q_tL^p_x$ are then obtained from the decay estimate via the standard $TT^{\ast}$ argument.
\begin{thm}\label{thm 2.4}
	For any $\alpha\in\lf[0,m-1\ri]$, a pair $\lf(p,q\ri)$ is called admissible if $\frac{1}{2mp}+\frac{1}{\lf(\alpha+1\ri)q}=\frac{1}{4m}$, where $p\in\lf[2,+\infty\ri]$, and $q\in\lf[\frac{4m}{\alpha+1},+\infty\ri]$, then, for every admissible pair $\lf(p,q\ri)$, one has,
	\be\label{2.11}
	\lVert D_{x}^{\frac{\alpha}{2}\left( 1-\frac{2}{p} \right)}S\left( t \right) u_0 \rVert _{L_{t}^{q}L_{x}^{p}}\les\lVert u_0 \rVert _{L_{x}^{2}}.
	\ee
\end{thm}
\begin{proof}
	First, we set $T:=D_{x}^{\frac{\alpha}{2}\left( 1-\frac{2}{p} \right)}S\left( t \right)$, its adjoint is given by $T^{\ast}=\int_{\mathbb{R}}{D_{x}^{\frac{\alpha}{2}\left( 1-\frac{2}{p} \right)}S\left( -\tau \right) d\tau}$. To prove \eqref{2.11}, it suffices to show that $T$ is bounded. By the standard $TT^{\ast}$ argument, this reduces to proving the boundedness of $T^{\ast}$ and $TT^{\ast}$. Consequently, \eqref{2.11} is equivalent to the following two estimates. 
	\be\label{2.12}
	\lVert T^*f \rVert _{L_{x}^{2}}=\Big\lVert\int_{\mathbb{R}}{D_{x}^{\frac{\alpha}{2}\left( 1-\frac{2}{p} \right)}S\left( -\tau \right) f\left( \tau ,\cdot \right) d\tau} \Big\rVert _{L_{x}^{2}}\les\lVert f \rVert _{L_{t}^{q^{\prime}}L_{x}^{p^{\prime}}},
	\ee
	\be\label{2.13}
	\lVert TT^*f \rVert _{L_{t}^{q}L_{x}^{p}}=\Big\lVert \int_{\mathbb{R}}{D_{x}^{\alpha \left( 1-\frac{2}{p} \right)}S\left( t-\tau \right) f\left( \tau ,\cdot \right) d\tau} \Big\rVert _{L_{t}^{q}L_{x}^{p}}\les\lVert f \rVert _{L_{t}^{q^{\prime}}L_{x}^{p^{\prime}}},
	\ee
where $\lf(p^{\prime},q^{\prime}\ri)$ is the conjugate pair of $\lf(p,q\ri)$. Thus, it suffices to show \eqref{2.13}. Recall that $S\left( t \right) u_0:=G\left( \cdot ,t \right)\ast u_0\left( \cdot \right)$, hence, by Young's inequality, one has,
\be\label{2.14}
\lVert D_{x}^{\alpha}S\left( t \right) u_0 \rVert _{L_{x}^{\infty}}=\lVert D_{x}^{\alpha}G\left( \cdot ,t \right)\ast u_0\left( \cdot \right) \rVert _{L_{x}^{\infty}}\le \lVert D_{x}^{\alpha}G\left( x,t \right) \rVert _{L_{x}^{\infty}}\lVert u_0 \rVert _{L^1_x}\les\left| t \right|^{-\frac{\alpha +1}{2m}}\lVert u_0 \rVert _{L^1_x}.
\ee
In view of the trivial $L^2$-isometry:
\be\label{2.15}
\lVert S\left(t\right)u_0\lVert_{L^{2}_x}=\lVert u_0\lVert_{L^2_x},
\ee
we define the analytic family of operators: $T_{\theta}:=\left| t \right|^{\frac{\theta \left( \alpha +1 \right)}{2m}}D_{x}^{\theta \alpha}S\left( t \right)$. Applying complex interpolation(see \cite{lofstrom1976interpolation}) between the lines $Re\left(\theta\right)=0$ and $Re\left(\theta\right)=1$, and using \eqref{2.14}, \eqref{2.15}, we have,
\be\label{2.16}
\lVert D_{x}^{\alpha \lf(1-\frac{2}{p}\ri)}S\left( t \right) u_0 \rVert _{L_{x}^{p}}\les\left| t \right|^{-\frac{ \alpha +1}{2m}\left( 1-\frac{2}{p}\right)}\lVert u_0 \rVert _{L_{x}^{p^{\prime}}}.
\ee Applying \eqref{2.16} together with the Hardy–Littlewood–Sobolev inequality yields,
\be\label{2.17}
\begin{aligned}
\Big\lVert \int_{\mathbb{R}}{D_{x}^{\alpha \left( 1-\frac{2}{p} \right)}S\left( t-\tau \right) f\left( \tau ,\cdot \right) d\tau} \Big\rVert _{L_{t}^{q}L_{x}^{p}}&\les\Big\lVert \int_{\mathbb{R}}{\lVert f\left( \tau ,\cdot \right) \rVert _{L_{x}^{p^{\prime}}}\left| t-\tau \right|}^{-\frac{\alpha +1}{2m}\left( 1-\frac{2}{p} \right)}d\tau \Big\rVert _{L_{t}^{q}}\\
&=\Big\lVert \lVert f\left( t,\cdot \right) \rVert _{L_{x}^{p^{\prime}}}\ast\left| t \right|^{-\frac{\alpha +1}{2m}\left( 1-\frac{2}{p} \right)} \Big\rVert _{L_{t}^{q}}\les\lVert f \rVert _{L_{t}^{q^{\prime}}L_{x}^{p^{\prime}}},
\end{aligned}
\ee
provided that $\frac{1}{q}=\frac{1}{q^{\prime}}-\big[1-\frac{\alpha +1}{2m}\left( 1-\frac{2}{p} \right)\big]$, which is precisely the condition $\frac{1}{2mp}+\frac{1}{\lf(\alpha+1\ri)q}=\frac{1}{4m}$ that we have already assumed. This completes the proof.
\end{proof}
For later use, we record the following special cases of  \eqref{2.11}.
\begin{col}One has,
\be\label{2.18}
\lVert S\left( t \right) u_0 \rVert _{L_{t}^{\infty}L_{x}^{2}}\les\lVert u_0 \rVert _{L_{x}^{2}},\qquad p=2,\ q=\infty,
\ee
\be\label{2.19}
\lVert D^{\frac{m-1}{2}}_xS\left( t \right) u_0 \rVert _{L_{t}^{4}L_{x}^{\infty}}\les\lVert u_0 \rVert _{L_{x}^{2}},\qquad \alpha=m-1,\ p=\infty,\ q=4,
\ee
\be\label{2.20}
\lVert S\left( t \right) u_0 \rVert _{L_{t,x}^{4m+2}}\les\lVert u_0 \rVert _{L_{x}^{2}},\qquad \alpha=0,\ p=q=4m+2, 
\ee
\be\label{2.21}
\lVert D_{x}^{\frac{m-1}{3}}S\left( t \right) u_0 \rVert _{L_{t,x}^{6}}\les\lVert u_0 \rVert _{L_{x}^{2}},\qquad \alpha=m-1,\ p=q=6.
\ee
\end{col}
\subsection{$L_{x}^{p}L_{t}^{q}$-type estimates}
\hspace{1em}
We now turn to the $L_{x}^{p}L_{t}^{q}$-type estimates for the same semigroup. The first estimate is the local smoothing effect,  for which the exponent $m-\frac{1}{2}$ is known to be sharp(see \cite{sun2017sharpness}). 
\begin{thm}\textbf{(Local smoothing effect)}\label{thm 2.6}
\be\label{2.22}
\lVert D_{x}^{m-\frac{1}{2}}S\left( t \right) u_0 \rVert _{L_{x}^{\infty}L_{t}^{2}}\les\lVert u_0 \rVert _{L_{x}^{2}}.
\ee
\end{thm}
\begin{proof}
	First, noting that $\eta=\eta \left( \xi \right)=\xi^{2m}$ is invertible on each of the intervals $\left( 0,+\infty \right)$ and $\left( -\infty ,0 \right)$ with inverse $\xi={\eta_1}^{-1}:={\eta_1}^{-1}(\xi)$ on $\left( 0,+\infty \right)$, and $\xi={\eta_2}^{-1}:={\eta_2}^{-1}(\xi)$ on $\left( -\infty ,0 \right)$, we can write
	\be
	\begin{aligned}
	S\left( t \right) u_0&=\int_0^{+\infty}{e^{^{-i\xi ^{2m}t}}e^{ix\xi}\widehat{u_0}\left( \xi \right) d\xi}+\int^0_{-\infty}{e^{^{-i\xi ^{2m}t}}e^{ix\xi}\widehat{u_0}\left( \xi \right) d\xi}\\
	&=\widehat{h_1}^{\left( \eta_1 \right)}-\widehat{h_2}^{\left( \eta_2 \right)},
\end{aligned}
	\ee
	where $h_1\left( \eta_1 \right) :=\chi _{\eta_1 \ge 0}\frac{e^{ix{\eta_1} ^{-1}}\widehat{u_0}\left( {\eta_1}^{-1} \right)}{{\eta_1}^{\prime}}$, $h_2\left( \eta_2 \right) :=\chi _{\eta_2 \ge 0}\frac{e^{ix{\eta_2} ^{-1}}\widehat{u_0}\left( {\eta_2}^{-1} \right)}{{\eta_2}^{\prime}}$, $\chi_A$ is the indicator function of the set $A$. Then, by Plancherel's equality, one has
	\be
	\lVert \widehat{h_1}^{\left( \eta_1 \right)} \rVert _{L_{t}^{2}}^{2}=\lVert h_1 \rVert _{L_{\eta_1}^{2}}^{2}\le \int_0^{+\infty}{\left| \frac{\widehat{u_0}\left( {\eta_1}^{-1} \right)}{{\eta_1}^{\prime}} \right|}^2d\eta_1 =\int_0^{+\infty}{\left| \frac{\widehat{u_0}\left( \xi \right)}{{\eta_1}^{\prime}\left( \xi \right)} \right|^2\left| {\eta_1}^{\prime}\left( \xi \right) \right|d\xi}\les\int_{\mathbb{R}}{\frac{\left| \widehat{u_0}\left( \xi \right) \right|^2}{\left| \xi \right|^{2m-1}}d\xi},
	\ee
	Similarly, we have $\lVert \widehat{h_2}^{\left( \eta_2 \right)} \rVert _{L_{t}^{2}}^{2}\les\int_{\mathbb{R}}{\frac{\left| \widehat{u_0}\left( \xi \right) \right|^2}{\left| \xi \right|^{2m-1}}d\xi}$. Finally, Minkowski's inequality yields
	\be
	\lVert S\left( t \right) u_0 \rVert _{L_{t}^{2}}\le \lVert \widehat{h_1}^{\left( \eta_1 \right)} \rVert _{L_{t}^{2}}+\lVert \widehat{h_2}^{\left( \eta_2 \right)} \rVert _{L_{t}^{2}}\les\lVert u_0 \rVert _{\dot{H}^{\frac{1}{2}-m}},
	\ee
	which is precisely \eqref{2.22}.
\end{proof}
We then turn to a Fourier restriction type estimate, which will also be needed in the sequel.
\begin{thm}
	\be\label{2.26}
	\lVert D_{x}^{-\frac{1}{4}}S\left( t \right) u_0 \rVert _{L_{x}^{4}L_{t}^{\infty}}\les\lVert u_0 \rVert _{L_{x}^{2}}.
	\ee
\end{thm}
\begin{proof}
	By Theorem 2.5 in \cite{kenig1991oscillatory}, we have
	\be\label{2.27}
	\lVert S\left( t \right) u_0 \rVert _{L_{x}^{4}L_{t}^{\infty}}^{2}\les\int_{\mathbb{R}}{\left| \widehat{u_0}\left( \xi \right) \right|^2\left| \frac{\eta ^{\prime}\left( \xi \right)}{\eta ^{\prime\prime}\left( \xi \right)} \right|}^{\frac{1}{2}}d\xi\les\lVert u_0 \rVert _{{\dot{H}}^{\frac{1}{4}}},
	\ee
	where $\eta=\xi^{2m}$, which is exactly \eqref{2.26}.
\end{proof}
\subsection{Fourier multiplier $F_{\rho}$}
\hspace{1em}
By the transference principle, the estimates obtained in the previous two subsections can be transferred to the $X^{s,b}$ framework(see \cite{tao2006nonlinear}). This is the key to obtaining a unified treatment of the multilinear estimates for different types of nonlinearities!
\begin{thm}
	For any $\rho>\frac{1}{2}$, one has
	\be\label{2.28}
	\lVert F_{\rho} \rVert _{L_{t}^{\infty}L_{x}^{2}}\les\lVert f \rVert _{L_{\tau,\xi}^{2}},
	\ee
	\be\label{2.29}
	\lVert D_{x}^{\frac{m-1}{2}}F_{\rho} \rVert _{L_{t}^{4}L_{x}^{\infty}}\les\lVert f \rVert _{L_{\tau,\xi}^{2}},
	\ee
	\be\label{2.30}
	\lVert F_{\rho} \rVert _{L_{t,x}^{4m+2}}\les\lVert f \rVert _{L_{\tau,\xi}^{2}},
	\ee
	\be\label{2.31}
	\lVert D_{x}^{\frac{m-1}{3}}F_{\rho} \rVert _{L_{t,x}^{6}}\les\lVert f \rVert _{L_{\tau,\xi}^{2}},
	\ee
	\be\label{2.32}
	\lVert D_{x}^{m-\frac{1}{2}}F_{\rho} \rVert _{L_{x}^{\infty}L_{t}^{2}}\les\lVert f \rVert _{L_{\xi ,\tau}^{2}},
	\ee
	\be\label{2.33}
	\lVert D_{x}^{-\frac{1}{4}}F_{\rho} \rVert _{L_{x}^{4}L_{t}^{\infty}}\les\lVert f \rVert _{L_{\xi ,\tau}^{2}},
	\ee
\end{thm}
\begin{proof}
	We can write $F_{\rho ,+}f=S\left( t \right) K_+$ and $F_{\rho ,-}f=S\left( -t \right) K_-$, where 
	\be\label{2.34}
	\widehat{K_+}^{\left( x \right)}\left( \xi ,t \right) =\int_{\mathbb{R}}{\frac{f\left( \xi ,\lambda -\xi ^{2m} \right)e^{it\lambda}}{\left< \lambda \right> ^{\rho}}}d\lambda,
	\ee
	\be\label{2.35}
	\widehat{K_-}^{\left( x \right)}\left( \xi ,t \right) =\int_{\mathbb{R}}{\frac{f\left( \xi ,\lambda +\xi ^{2m} \right)e^{it\lambda}}{\left< \lambda \right> ^{\rho}}}d\lambda .
	\ee
	Then, applying Cauchy-Schwarz inequality together with \eqref{2.18}--\eqref{2.22}, and \eqref{2.26} yields the desired estimates. We only proof \eqref{2.28} for $F_{\rho ,+}$; the proof for $F_{\rho ,-}$ and \eqref{2.29}-\eqref{2.33} are analogous and are omitted. Using \eqref{2.34}, and \eqref{2.18}, one has
	\ba
	\lVert F_{\rho ,+}f \rVert _{L_{t}^{\infty}L_{x}^{2}}^{2}=\lVert S\left( t \right) K_+ \rVert _{L_{t}^{\infty}L_{x}^{2}}^{2}\les\lVert K_+ \rVert _{L_{x}^{2}}^{2}&\sim\lVert \int_{\mathbb{R}}{\frac{f\left( \xi ,\lambda -\xi ^{2m} \right) e^{it\lambda}}{\left< \lambda \right> ^{\rho}}}d\lambda \rVert _{L_{\xi}^{2}}^{2}\\
	&\le\Big\lVert \left( \int_{\mathbb{R}}{\left< \lambda \right> ^{-2\rho}d\lambda} \right) ^{\frac{1}{2}}\lVert f\left( \xi ,\lambda -\xi ^{2m} \right) \rVert _{L_{\lambda}^{2}} \Big\rVert _{L_{\xi}^{2}}^{2}\les\lVert f \rVert _{L_{\tau ,\xi}^{2}}^{2},
	\ea
	where we have taken advantage of the fact that $\rho>\fr{1}{2}$ and the change of variable $\tau=\lambda -\xi ^{2m}$.
\end{proof}
\par
Note that Plancherel's equality gives the trivial estimate
\be\label{2.36}
\lVert F_0 \rVert _{L_{t,x}^{2}}\les\lVert f \rVert _{L_{\tau ,\xi}^{2}}.
\ee
Then, by interpolation between \eqref{2.30} and \eqref{2.36}, we obtain that for any $\rho>\frac{2m+1}{2m}\left(\frac{1}{2}-\frac{1}{p}\right)$, one has
\be\label{2.37}
\lVert F_{\rho} \rVert _{L_{t,x}^{p}}\les\lVert f \rVert _{L_{\tau ,\xi}^{2}},
\ee
where $2\le p\le 4m+2$. Particularly, one has the following Corollary.
\begin{col} For any $\rho>\frac{2m+1}{12m}$, we have
	\be\label{2.38}
	\lVert F_{\rho} \rVert _{L_{t,x}^{3}}\les\lVert f \rVert _{L_{\tau ,\xi}^{2}},
	\ee
	for any $\rho>\frac{2m+1}{8m}$, we have
	\be\label{2.39}
	\lVert F_{\rho} \rVert _{L_{t,x}^{4}}\les\lVert f \rVert _{L_{\tau ,\xi}^{2}},
	\ee
	for any $\rho>\frac{2m+1}{6m}$, we have
	\be\label{2.40}
	\lVert F_{\rho} \rVert _{L_{t,x}^{6}}\les\lVert f \rVert _{L_{\tau ,\xi}^{2}}.
	\ee
\end{col}
Finally, we turn to a low-frequency estimate, which will be used to treat the low-frequency regime(i.e. when $\left|\xi\right|$ is small); this complements the high-frequency estimates already established in the preceding proof.
\begin{col}
	For any fixed $N>0$, let us define the frequency projection operator $P_N$ by 
	\be\label{2.41}
	\widehat{P_Nf}\doteq\chi _{\left| \xi \right|\le N}\widehat{f}\left( \xi \right).
	\ee
	Then, for any $\rho>\frac{1}{2}$, we have the following estimate,
	\be\label{2.42}
	\lVert P_NF_{\rho} \rVert _{L_{x}^{2}L_{t}^{\infty}}\les\lVert u_0 \rVert _{L_{x}^{2}}.
	\ee
	\begin{proof}
		This follows directly from the Sobolev inequality; for details, see Lemma 2.2 of \cite{kenig1993cauchy}.
	\end{proof}
\end{col}
\section{Proof of Trilinear estimates}\label{sec 3}
\hspace{1em}
With the Strichartz estimates established above, we now turn to prove the key trilinear estimates. First, for $k=0,1,2,3$, we write the $X^{s,b}$
estimates in a unified form for later convenience. Unlike Tao's $[k;Z]$-multiplier norm method, our main proof can avoid the resonance effects caused by different choices of conjugation.
\subsection{Equivalent forms of the multilinear estimates}
\hspace{1em}
Following the idea of Tao\cite{tao2001multilinear}, let us first introduce a convolution hyperplane $K_4$ defined as follows,
\be\label{3.1}
K_4\doteq\left\{ \left( \vec{\xi},\vec{\tau} \right) \in \mathbb{R}^4\times \mathbb{R}^4|\sum_{j=1}^4{\xi _j}=\sum_{j=1}^4{\tau _j}=0 \right\}.
\ee
In addition, based on the format of nonlinearities, we introduce the following notations.
\be\label{3.2}
v_j\doteq\left\{ \begin{array}{l}
	\overline{u_j},\ if\ j\le k,\\
	u_j,\ if\ j>k,\\ 
\end{array} \right. \quad
L_j\left( \xi _j,\tau _j \right) \doteq\left\{ \begin{array}{l}
	\tau _j-\xi _j^{2m},\,\,if\,\,j\le k,\\
	\tau _j+\xi _j^{2m},\,\,if\,\,j>k,\\
\end{array} \right. 
\ee
where $j=1,2,3$. Moreover, we set $L_4\left( \xi _4,\tau _4 \right) =\tau _4-\xi _4^{2m}$. By duality, one has,
\be\label{3.3}
\begin{aligned}
\lVert N_k\left( u_1,u_2,u_3 \right) \rVert _{X^{s,-b}}=&\underset{\lVert v \rVert _{X^{-s,b}=1}}{sup}\left| \int_{\mathbb{R}^2}{v_1v_2v_3\overline{v}dxdt} \right|\\
=&\underset{\lVert v \rVert _{X^{-s,b}=1}}{sup}\left| \int_{K_4}{\widehat{v_1}\left( \xi _1,\tau _1 \right) \widehat{v_2}\left( \xi _2,\tau _2 \right) \widehat{v_3}\left( \xi _3,\tau _3 \right) \widehat{\overline{v}}\left( \xi _4,\tau _4 \right)} \right|.
\end{aligned}
\ee
Now we define,
\be\label{3.4}
f_j\left( \xi _j,\tau _j \right) \doteq \left\{ \begin{array}{l}
	\widehat{u_j}\left( \xi _j,\tau _j \right) \left< \xi _j \right> ^{s}\left< L_j \right> ^{b^{\prime}},\,\,if\,\,0\le j\le 3,\\
	\widehat{v}\left( \xi _4,\tau _4 \right) \left< \xi _4 \right> ^{-s}\left< L_4 \right> ^{b},\,\,if\,\,j=4.\\
\end{array} \right.
\ee
With this notation, the spatial Bourgain spaces estimates \eqref{1.16} are equivalent to
\be\label{3.5}
\varGamma _1\doteq\left| \int_{K_4}{\frac{\prod_{j=1}^3{\left< \xi _j \right> ^{-s}}\prod_{j=1}^4{f_j\left( \xi _j,\tau _j \right)}}{\left< \xi _4 \right> ^{-s}\left< L_4 \right> ^{b}\prod_{j=1}^3{\left< L_j \right> ^{b^{\prime}}}}} \right|\les\prod_{j=1}^4{\lVert f_j \rVert _{L_{\xi ,\tau}^{2}}},
\ee
$f_j$ are some $L^2$ test functions. Similarly, the temporal Bourgain spaces estimates \eqref{1.17} are equivalent to
\be\label{3.6}
\varGamma _2\doteq\left| \int_{K_4}{\frac{\prod_{j=1}^3{\left< \xi _j \right> ^{-s}}\prod_{j=1}^4{f_j\left( \xi _j,\tau _j \right)}}{\left< \tau _4 \right> ^{\frac{-s}{2m}}\left< L_4 \right> ^{b}\prod_{j=1}^3{\left< L_j \right> ^{b^{\prime}}}}} \right|\les\prod_{j=1}^4{\lVert f_j \rVert _{L_{\xi ,\tau}^{2}}}.
\ee
We now proceed to prove \eqref{3.5}. 
\subsection{Proof of the spatial Bourgain spaces estimates \eqref{3.5}}
\hspace{1em}
\begin{proof}[(\textbf{Proof of \eqref{3.5} for} {\boldmath$k=0,1,2$} \textbf{and} {\boldmath$m\ge2$})] The proof for $m=1$ is placed in Remark \ref{rem 3.2}. We want to point out that the conjugations in the nonlinearities alter the resonance function(see \eqref{3.23}). And this makes huge difference in Tao's $[k;Z]$-multiplier norm method, which causes difficulties in the analysis of Himonas and Yan\cite{himonas2024schrodinger,himonas2026higher}. In contrast, despite the changes in the resonance relations, the Strichartz estimates remain unchanged, indeed the multipliers $F_{\rho,+}$, $F_{\rho,-}$ satisfy the same estimates, which allows us to treat different nonlinearities in a unified manner. Therefore, we may relabel the frequency variables and assume, without loss of generality, that $\left| \xi _1 \right|\le \left| \xi _2 \right|\le \left| \xi _3 \right|$. Next, in order to estimate \eqref{3.5}, we now decompose the integration domain $K_4$ into several cases.
\begin{itemize}
	\item \textbf{Case A.} {\boldmath $\left| \xi _3 \right|\le 10$,}
	\item \textbf{Case B.} {\boldmath $\left| \xi _3 \right|\ge 10, \left| \xi _2 \right|\le 4$,}
	\item \textbf{Case C.} {\boldmath $\left| \xi _3 \right|\ge 10, \left| \xi _2 \right|\ge 4, \left| \xi _1 \right|\le 1$,}
	\item \textbf{Case D.} {\boldmath $\left| \xi _3 \right|\ge 10, \left| \xi _2 \right|\ge 4, \left| \xi _1 \right|\ge 1, \left| \xi _4 \right|\ge \frac{\left| \xi _1 \right|}{5}$,}
	\item \textbf{Case E.} {\boldmath $\left| \xi _3 \right|\ge 10, \left| \xi _2 \right|\ge 4, \left| \xi _1 \right|\ge 1, \left| \xi _4 \right|\le \frac{\left| \xi _1 \right|}{5}$.}
\end{itemize}
Furthermore, we will need to consider the following four subcases for some cases.
{\boldmath\be\label{3.7}
\begin{aligned}
&\left( a \right) \ \left< L_1 \right> =\underset{j=1,2,3,4}{\max}\left< L_j \right> ,\quad \left( b \right) \ \left< L_2 \right> =\underset{j=1,2,3,4}{\max}\left< L_j \right> ,\\ &\left( c \right) \ \left< L_3 \right> =\underset{j=1,2,3,4}{\max}\left< L_j \right> ,\quad \left( d \right) \ \left< L_4 \right> =\underset{j=1,2,3,4}{\max}\left< L_j \right> .
\end{aligned}
\ee}
In each of these four subcases, we shall apply the modulation exponent transfer technique to redistribute the exponents among the modulation factors, $L_j$, $j=1,2,3,4$, so that the estimate can be handled by the Strichartz estimates obtained before. For convenience, we set,
\be
\varGamma _{1,A}\doteq\left| \int_{K_4\cap A}{\frac{\prod_{j=1}^3{\left< \xi _j \right> ^{-s}}\prod_{j=1}^4{f_j\left( \xi _j,\tau _j \right)}}{\left< \xi _4 \right> ^{-s}\left< L_4 \right> ^{b}\prod_{j=1}^3{\left< L_j \right> ^{b^{\prime}}}}} \right|,
\ee
and define $\varGamma _{1,B}$, $\varGamma _{1,C}$, $\varGamma _{1,D}$, $\varGamma _{1,E}$ analogously. In addition, we denote,
\be\label{3.9}
M\doteq\frac{\prod_{j=1}^3{\left< \xi _j \right> ^{-s}}}{\left< \xi _4 \right> ^{-s}},
\ee
and write $\widehat{F_{\rho}^{\left( j \right)}}\doteq\widehat{F_{\rho}f_j}\left( \xi _j,\tau _j \right) $. In contrast to Himonas and Yan, who required $s>s_{m,1}$, we prove the result for $s\ge s_{m,k}$ with $k=0,1,2$. In addition, based on the relation $\sum_{j=1}^4{\xi _j}=0$, one has $\left< \xi _4 \right> \les\prod_{j=1}^3{\left< \xi _j \right>}$, thus, we have,
$$\left( \frac{\prod_{j=1}^3{\left< \xi _j \right>}}{\left< \xi _4 \right>} \right) ^{-s}=\left( \frac{\prod_{j=1}^3{\left< \xi _j \right>}}{\left< \xi _4 \right>} \right) ^{-s_{m,k}}\left( \frac{\left< \xi _4 \right>}{\prod_{j=1}^3{\left< \xi _j \right>}} \right) ^{s-s_{m,k}}\les\left( \frac{\prod_{j=1}^3{\left< \xi _j \right>}}{\left< \xi _4 \right>} \right) ^{-s_{m,k}}.$$
 Hence, it suffices to establish the estimate at the critical regularity $s=s_{m,k}$=$-\fr{m-1}{2}$.
\begin{itemize}
	\item \textbf{Case A.} {\boldmath $\left| \xi _3 \right|\le 10$}\quad In this case, we have the bound $M\les\prod_{j=1}^3{\left< \xi _j \right>}\les1$. Using Hölder's inequality and applying the Strichartz estimate \eqref{2.39}, we obtain,
	\be\label{3.10}
	\varGamma _{1,A}\les\prod_{j=1}^3{\lVert F_{b^{\prime}}^{\left( j \right)} \rVert _{L_{t,x}^{4}}}\lVert F_{b}^{\left( 4 \right)} \rVert _{L_{t,x}^{4}}\les\prod_{j=1}^4{\lVert f_j \rVert _{L_{\tau,\xi}^{2}}},
	\ee
	provided that,
	\be\label{3.11}
	b^{\prime}>\frac{2m+1}{8m},\quad b>\frac{2m+1}{8m}.
	\ee
	\item \textbf{Case B.} {\boldmath $\left| \xi _3 \right|\ge 10, \left| \xi _2 \right|\le 4$}\quad First, since $\left| \xi _4 \right|\ge \left| \xi _3 \right|-\left| \xi _1 \right|-\left| \xi _2 \right|\ge \frac{\left| \xi _3 \right|}{5}$, one has $\left| \xi _4 \right|\sim\left| \xi _3 \right|$. We then have the bound $M\les1$, and the estimate follows as in  \textbf{Case A}, provided that \eqref{3.11} holds.
	\item \textbf{Case C.} {\boldmath $\left| \xi _3 \right|\ge 10, \left| \xi _2 \right|\ge 4, \left| \xi _1 \right|\le 1$} In this case, we first consider, $$M\les\frac{\prod_{j=1}^3{\left< \xi _j \right> ^{\frac{m-1}{2}}}}{\left< \xi _4 \right> ^{\frac{m-1}{2}}}\les\prod_{j=1}^3{\left< \xi _j \right> ^{\frac{m-1}{2}}}.$$
	\subitem \textbf{Case C.(a)}\quad {\boldmath$\left< L_1 \right> =\underset{j=1,2,3,4}{\max}\left< L_j \right>$}\quad From the frequency ordering and a modulation exponent transfer, we have,
	\be\label{3.12}
	M\les\left< \xi _2 \right> ^{\frac{m-1}{2}}\left< \xi _3 \right> ^{\frac{m-1}{2}}\sim\left| \xi _2\xi _3 \right|^{\frac{m-1}{2}},\quad \left< L_4 \right> ^{b}\prod_{j=1}^3{\left< L_j \right> ^{b^{\prime}}}\gtr\ \prod_{j=2}^4{\left< L_j \right> ^{b^{\prime}+\frac{b}{3}}}.
	\ee
	Combining the above two estimates and Strichartz estimates \eqref{2.28}, \eqref{2.29} yields,
	\be\label{3.13}
	\varGamma _{1,C}\les\lVert F_{0}^{\left( 1 \right)} \rVert _{L_{t,x}^{2}}\lVert D_{x}^{\frac{m-1}{2}}F_{b^{\prime}+\frac{b}{3}}^{\left( 2 \right)} \rVert _{L_{t}^{4}L_{x}^{\infty}}\lVert D_{x}^{\frac{m-1}{2}}F_{b^{\prime}+\frac{b}{3}}^{\left( 3 \right)} \rVert _{L_{t}^{4}L_{x}^{\infty}}\lVert F_{b^{\prime}+\frac{b}{3}}^{\left( 4 \right)} \rVert _{L_{t}^{\infty}L_{x}^{2}}\les\prod_{j=1}^4{\lVert f_j \rVert _{L_{\xi ,\tau}^{2}}},
	\ee
	provided that,
	\be\label{3.14}
	b^{\prime}+\frac{b}{3}>\frac{1}{2}.
	\ee
	\textbf{Case C.(b)} {\boldmath$\left< L_2 \right> =\underset{j=1,2,3,4}{\max}\left< L_j \right>$}\quad First, we have the bound,
	\be\label{3.15}
	M\les\left< \xi _2 \right> ^{\frac{m-1}{2}}\left< \xi _3 \right> ^{\frac{m-1}{2}}\les\left| \xi _3 \right|^{m-1}\les\left| \xi _3 \right|^{m-\frac{1}{2}}.
	\ee In this subcase, we perform a modulation exponent transfer by rewriting the denominator as,
	\be\label{3.16}
	\left< L_4 \right> ^{b}\prod_{j=1}^3{\left< L_j \right> ^{b^{\prime}}}\gtr\left< L_1 \right> ^{b^{\prime}+\epsilon}\left< L_2 \right> ^{b^{\prime}-2\epsilon}\left< L_3 \right> ^{b^{\prime}+\epsilon}\left< L_4 \right> ^{b},\qquad with\ \epsilon=\frac{1}{9}-\frac{1}{36m}.
	\ee
	Combining this with the Strichartz estimates \eqref{2.32}, \eqref{2.38}, \eqref{2.40} and \eqref{2.42}, we obtain,
	\be\label{3.17}
	\varGamma _{1,C}\les\lVert P_1F_{b^{\prime}+\epsilon}^{\left( 1 \right)} \rVert _{L_{x}^{2}L_{t}^{\infty}}\lVert F_{b^{\prime}-2\epsilon}^{\left( 2 \right)} \rVert _{L_{x,t}^{3}}\lVert D_{x}^{m-\frac{1}{2}}F_{b^{\prime}+\epsilon}^{\left( 3 \right)} \rVert _{L_{x}^{\infty}L_{t}^{2}}\lVert F_{b}^{\left( 4 \right)} \rVert _{L_{x,t}^{6}}\les\prod_{j=1}^4{\lVert f_j \rVert _{L_{\xi ,\tau}^{2}}},
	\ee
	where $P_1$ is the  projection defined by \eqref{2.37}, provided that,
	\be\label{3.18}
	b^{\prime}-2\epsilon>\frac{2m+1}{12m},\quad
	b^{\prime}+\epsilon>\frac{1}{2},\quad b>\frac{2m+1}{6m}.
	\ee\\ 
	\textbf{Case C.(c)} {\boldmath$\left< L_3 \right> =\underset{j=1,2,3,4}{\max}\left< L_j \right>$}\quad We distinguish two subcases according to the size of $\xi_4$.
	\subitem \textbf{Case.C.(c).1:} {\boldmath $\left| \xi _4 \right|\ge \frac{\left| \xi _3 \right|}{5}$}. In this case $\left| \xi _4 \right|\sim\left| \xi _3 \right|$, and we have the following estimate and modulation exponent transfer,
	\be\label{3.19}
	M\les\left| \xi _4 \right|^{m-\frac{1}{2}},\qquad \left< L_4 \right> ^{b}\prod_{j=1}^3{\left< L_j \right> ^{b^{\prime}}}\gtr\left< L_1 \right> ^{b^{\prime}+\epsilon}\left< L_2 \right> ^{b}\left< L_3 \right> ^{b^{\prime}-2\epsilon}\left< L_4 \right> ^{b^{\prime}+\epsilon},
	\ee
	which yields,
	 \be\label{3.20}
	 \varGamma _{1,C}\les\lVert P_1F_{b^{\prime}+\epsilon}^{\left( 1 \right)} \rVert _{L_{x}^{2}L_{t}^{\infty}}\lVert F_{b}^{\left( 2 \right)} \rVert _{L_{x,t}^{6}}\lVert F_{b^{\prime}-2\epsilon}^{\left( 3 \right)} \rVert _{L_{x,t}^{3}}\lVert D_{x}^{m-\frac{1}{2}}F_{b^{\prime}+\epsilon}^{\left( 4 \right)} \rVert _{L_{x}^{\infty}L_{t}^{2}}\les\prod_{j=1}^4{\lVert f_j \rVert _{L_{\xi ,\tau}^{2}}},
	 \ee
	 provided that \eqref{3.18}, $b>b^{\prime}$ and $b^{\prime}+\epsilon>b$ hold. We point out that $b^{\prime}+\epsilon>b$ is automatically satisfied since $b^{\prime}+\epsilon>\fr{1}{2}>b$. 
	\subitem \textbf{Case.C.(c).2:} {\boldmath $\left| \xi _4 \right|\le \frac{\left| \xi _3 \right|}{5}$}. Then $\left| \xi _2 \right|\ge \left| \xi _3 \right|-\left| \xi _1 \right|-\left| \xi _4 \right|\ge \frac{7}{10}\left| \xi _3 \right|$, so that $\left| \xi _2 \right|\sim\left| \xi _3 \right|$, using the assumption that $\left| \xi _1 \right|\le 1$. The estimate then follows by the similar argument leading to \eqref{3.20}.\\
	\textbf{Case C.(d)} {\boldmath$\left< L_4 \right> =\underset{j=1,2,3,4}{\max}\left< L_j \right>$}\quad This subcase can be handled similarly to \textbf{Case C.(a)} by noting that $M\les\left| \xi _2\xi _3 \right|^{\frac{m-1}{2}}$ and $\left< L_4 \right> ^{b}\prod_{j=1}^3{\left< L_j \right> ^{b^{\prime}}}\gtr\ \prod_{j=1}^3{\left< L_j \right> ^{b^{\prime}+\frac{b}{3}}}$. Thus, $\varGamma _{1,C}$ can be estimated in the same way as in \textbf{Case C.(a)}, provided that \eqref{3.14} holds.
	\item \textbf{Case D.} {\boldmath $\left| \xi _3 \right|\ge 10, \left| \xi _2 \right|\ge 4, \left| \xi _1 \right|\ge 1, \left| \xi _4 \right|\ge \frac{\left| \xi _1 \right|}{5}$} In this case, we have $\left| \xi _4 \right|\sim\left| \xi _3 \right|$ or $\left| \xi _2 \right|\sim\left| \xi _3 \right|$. More precisely, if $|\xi_4|\ge\fr{|\xi_2|}{4}$, then one has $|\xi_4|\ge|\xi_3|-|\xi_2|-|\xi_1|\ge|\xi_3|-4|\xi_4|-5|\xi_4|$, which implies that $|\xi_4|\sim|\xi_3|$. Similarly, if $|\xi_4|\le\fr{|\xi_2|}{4}$, then we have $|\xi_2|\ge|\xi_3|-|\xi_1|-|\xi_4|\ge|\xi_3|-|\xi_2|-\fr{|\xi_4|}{4}$, which will lead to $|\xi_2|\sim|\xi_3|$.
	
	\subitem \textbf{Case D.(a)} {\boldmath$\left< L_1 \right> =\underset{j=1,2,3,4}{\max}\left< L_j \right>$}\quad We observe that $M\les\left< \xi _2 \right> ^{\frac{m-1}{2}}\left< \xi _3 \right> ^{\frac{m-1}{2}}\les\left| \xi _2\xi _3 \right|^{\frac{m-1}{2}}$. This is exactly the same bound as in \textbf{Case C.(a)}. Hence the estimate for $\varGamma_{1,D}$ follows by the same argument, and we omit the details.
	\subitem \textbf{Case D.(b)} {\boldmath$\left< L_2 \right> =\underset{j=1,2,3,4}{\max}\left< L_j \right>$}\quad In this subcase, we have 
	\be\label{3.21}
	M\les\left< \xi _3 \right> ^{m-1}\les\frac{\left| \xi _3 \right|^{m-\frac{1}{2}}}{\left| \xi _1\xi _4 \right|^{\frac{1}{4}}},\quad \left< L_4 \right> ^{b}\prod_{j=1}^3{\left< L_j \right> ^{b^{\prime}}}\gtr\prod_{j=1,3,4}{\left< L_j \right> ^{b^{\prime}+\frac{b}{3}}},
	\ee
	where we used $\left| \xi _4 \right|\gtr\left| \xi _1 \right|$. Applying the Strichartz estimates \eqref{2.32}, \eqref{2.33} yields
	\be\label{3.22}
	\varGamma _{1,D}\les\lVert D_{x}^{-\frac{1}{4}}F_{b^{\prime}+\frac{b}{3}}^{\left( 1 \right)} \rVert _{L_{x}^{4}L_{t}^{\infty}}\lVert F_{0}^{\left( 2 \right)} \rVert _{L_{x,t}^{2}}\lVert D_{x}^{m-\frac{1}{2}}F_{b^{\prime}+\frac{b}{3}}^{\left( 3 \right)} \rVert _{L_{x}^{\infty}L_{t}^{2}}\lVert D_{x}^{-\frac{1}{4}}F_{b^{\prime}+\frac{b}{3}}^{\left( 4 \right)} \rVert _{L_{x}^{4}L_{t}^{\infty}}\les\prod_{j=1}^4{\lVert f_j \rVert _{L_{\xi ,\tau}^{2}}},
	\ee
	provided that \eqref{3.14} holds.\\
	 \subitem \textbf{Case D.(c)} {\boldmath$\left< L_3 \right> =\underset{j=1,2,3,4}{\max}\left< L_j \right>$}\quad This subcase can be treated similarly to \textbf{Case C.(b)}, with the observation that, $\left| \xi _4 \right|\sim\left| \xi _3 \right|$ or $\left| \xi _2 \right|\sim\left| \xi _3 \right|$, which allows the transfer of exponents between the two factors. 
	 \subitem \textbf{Case D.(d)} {\boldmath$\left< L_4 \right> =\underset{j=1,2,3,4}{\max}\left< L_j \right>$}\quad In this subcase, the argument is again analogous to \textbf{Case D.(a)}.
	\item \textbf{Case E.} {\boldmath $\left| \xi _3 \right|\ge 10, \left| \xi _2 \right|\ge 4, \left| \xi _1 \right|\ge 1, \left| \xi _4 \right|\le \frac{\left| \xi _1 \right|}{5}$.}\quad In this case, one has
	$$M\les\frac{\prod_{j=1}^3{\left< \xi _j \right> ^{\frac{m-1}{2}}}}{\left< \xi _4 \right> ^{\frac{m-1}{2}}}\les\prod_{j=1}^3{\left< \xi _j \right> ^{\frac{m-1}{2}}}.$$
	and $|\xi_2|\sim|\xi_3|$, due to the fact that $|\xi_2|\ge|\xi_3|-|\xi_1|-|\xi_4|\ge|\xi_3|-|\xi_1|-\fr{|\xi_1|}{4}\ge|\xi_3|-\fr{5|\xi_2|}{4}$.\par
	Then, we will need the help from the resonance function $H_k$, with $k=0,1,2$, defined as follows.
	\be\label{3.23}
	H_k\doteq\sum_{j=1}^4{L_j}\left(k\right)=\left\{ \begin{array}{ll}
		\xi _{1}^{2m}+\xi _{2}^{2m}+\xi _{3}^{2m}-\xi _{4}^{2m},\quad &k=0,\\
		-\xi _{1}^{2m}+\xi _{2}^{2m}+\xi _{3}^{2m}-\xi _{4}^{2m},\quad &k=1,\\
		-\xi _{1}^{2m}-\xi _{2}^{2m}+\xi _{3}^{2m}-\xi _{4}^{2m},\quad &k=2,\\
	\end{array} \right. 
	\ee
	For each $k$, we have the following bound.
	\begin{cla}\label{cla 3.1}
		Under the assumptions of \textbf{Case E}, the resonance function satisfies the lower bound
		\be\label{3.24}
		\left| H_k \right|\ge \left( 2m+\left( \frac{4}{5} \right) ^{2m}-\frac{26}{25} \right) \left| \xi _2 \right|^{2m-1}\left| \xi _1 \right|,\qquad k=0,1,2.
		\ee
	\end{cla}
\textbf{Proof of Claim \eqref{cla 3.1}} First, under the assumptions of \textbf{Case E} and recall that $\sum_{j=1}^4{\xi _j}=0$, $\left| \xi _1 \right|\le \left| \xi _2 \right|\le \left| \xi _3 \right|$, we have 
	\be\label{3.25}
	\left| \xi _3 \right|=\left| \xi _2 \right|+\left| \xi _1 \right|\pm \left| \xi _4 \right|\ge\left| \xi _2 \right|+\frac{4}{5}\left| \xi _1 \right|.
	\ee
	Consequently, we obtain the following lower bound for the resonance function.
	\be\label{3.26}
	\begin{aligned}
		\left| H_k\right|\ge \left| \xi _3 \right|^{2m}-\left| \xi _1 \right|^{2m}-\left| \xi _2 \right|^{2m}-\left| \xi _4 \right|^{2m}
		&\ge \left| \xi _3 \right|^{2m}-\left| \xi _2 \right|^{2m}-\frac{26}{25}\left| \xi _1 \right|^{2m}\\
		&\ge \left( \left| \xi _2 \right|+\frac{4}{5}\left| \xi _1 \right| \right) ^{2m}-\left| \xi _2 \right|^{2m}-\frac{26}{25}\left| \xi _1 \right|^{2m}\\
		&\ge 2m\left| \xi _2 \right|^{2m-1}\left| \xi _1 \right|+\left( \frac{4}{5} \right) ^{2m}\left| \xi _1 \right|^{2m}-\frac{26}{25}\left| \xi _1 \right|^{2m}\\
		&\ge \left( 2m+\left( \frac{4}{5} \right) ^{2m}-\frac{26}{25} \right) \left| \xi _2 \right|^{2m-1}\left| \xi _1 \right|.
	\end{aligned}
\ee
With this claim in hand, we now proceed to estimate $\varGamma _{1,E}$.
\subitem \textbf{Case E.(a)} {\boldmath$\left< L_1 \right> =\underset{j=1,2,3,4}{\max}\left< L_j \right>$}\quad In this subcase, the modulation exponent transfer gives,
\be\label{3.27}
\left< L_4 \right> ^{b}\prod_{j=1}^3{\left< L_j \right> ^{b^{\prime}}}\gtr\left< H_k \right> ^{\frac{m-1}{4m}}\prod_{j=2,3,4}{\left< L_j \right> ^{b^{\prime}+\frac{b}{3}-\frac{m-1}{12m}}},
\ee
where we used the relation $\left| H_k \right|\le \sum_{j=1}^4{\left| L_j \right|}\les\left| L_1 \right|$. For the multiplier $M$ we have,
\be\label{3.28}
\begin{aligned}
\frac{M}{\left< L_4 \right> ^b\prod_{j=1}^3{\left< L_j \right> ^{b^{\prime}}}}\les\frac{M}{\left< H_k \right> ^{\frac{m-1}{4m}}\prod_{j=2,3,4}{\left< L_j \right> ^{b^{\prime}+\frac{b}{3}-\frac{m-1}{12m}}}}&\les\frac{\left| \xi _1\xi _2\xi _3 \right|^{\frac{m-1}{2}}}{\left< H_k \right> ^{\frac{m-1}{4m}}\prod_{j=2,3,4}{\left< L_j \right> ^{b^{\prime}+\frac{b}{3}-\frac{m-1}{12m}}}}\\&\les\frac{\left| \xi _3 \right|^{m-\frac{1}{2}}}{\left| \xi _2\xi _4 \right|^{\frac{1}{4}}\prod_{j=2,3,4}{\left< L_j \right> ^{b^{\prime}+\frac{b}{3}-\frac{m-1}{12m}}}}.
\end{aligned}
\ee Combining these estimates and proceeding as in \textbf{Case D.(b)}, we obtain,
\be\label{3.29}
\varGamma _{1,E}\les\lVert F_{0}^{\left( 1 \right)} \rVert _{L_{x,t}^{2}}\lVert D_{x}^{-\frac{1}{4}}F_{\sigma_1}^{\left( 2 \right)} \rVert _{L_{x}^{4}L_{t}^{\infty}}\lVert D_{x}^{m-\frac{1}{2}}F_{\sigma_1}^{\left( 3 \right)} \rVert _{L_{x}^{\infty}L_{t}^{2}}\lVert D_{x}^{-\frac{1}{4}}F_{\sigma_1}^{\left( 4 \right)} \rVert _{L_{x}^{4}L_{t}^{\infty}}\les\prod_{j=1}^4{\lVert f_j \rVert _{L_{\xi ,\tau}^{2}}},
\ee
subject to the conditions,
\be\label{3.30}
\sigma_1\doteq b^{\prime}+\frac{b}{3}-\frac{m-1}{12m}>\frac{1}{2}.
\ee
The remaining subcases \textbf{E.(b)}, \textbf{E.(c)}, and \textbf{E.(d)} can be estimated similarly; we omit the details. We note that in \textbf{E.(c)}, the relation $\left| \xi _2 \right|\sim\left| \xi _3 \right|$ will be used to facilitate the exponent transfer.
\end{itemize}
To close the proof, we summarize the restrictions on the parameters. Collecting all the conditions \eqref{3.11}, \eqref{3.14}, \eqref{3.18}, and \eqref{3.30} from the cases above and recall $\epsilon=\frac{1}{9}-\frac{1}{36m}$, we find that \eqref{3.5} holds for any $b'<b$ satisfying,
\be\label{3.31}
\frac{7}{18}+\frac{1}{36m}<b^{\prime}<\fr{1}{2},\quad \fr{1}{3}+\fr{1}{6m}<b<\fr{1}{2},\quad b^{\prime}+\fr{b}{3}>\fr{7}{12}-\fr{1}{12m},
\ee
The proof is now complete.\end{proof}
\begin{rem}\label{rem 3.2}
	For the special case $m=1$, it suffices to take $\fr{3}{8}<b,b^{\prime}<\fr{1}{2}$. In fact, in this case \eqref{3.5} can be estimated directly by the following $L^4$ Strichartz estimate.
	\be\label{3.32}
	\varGamma _1\les\prod_{j=1}^3{\lVert F_{b^{\prime}}^{\left( j \right)} \rVert _{L_{t,x}^{4}}}\lVert F_{b}^{\left( 4 \right)} \rVert _{L_{t,x}^{4}}.
	\ee
	The estimate \eqref{3.6} can be handled similarly since we have used the bound $\left< \xi _4 \right> ^{-\frac{m-1}{2}}\les1$. This provides an alternative proof of the result of Himonas and Yan\cite{himonas2024schrodinger}.
\end{rem}
 Next, we proceed to prove the case $k=3$. We continue to use the notation from \eqref{3.7}, \eqref{3.9} and assume as before that $\left| \xi _1 \right|\le \left| \xi _2 \right|\le \left| \xi _3 \right|$. Let $p\doteq-s$. The nonlinearity $\bar u_1\bar u_2\bar u_3$ exhibits a stronger resonance, which allows the estimates to hold below $s_{m,1}$. More precisely, it can be checked that the resonance function $H_3$, defined in \eqref{3.23}, satisfies,
\be\label{3.33}
\left| H_3 \right|=\sum_{j=1}^4{\xi _j^{2m}}\gtrsim\underset{j=1,2,3,4}{\max}\xi _j^{2m}\gtrsim\xi _3^{2m}.
\ee This lower bound is sufficiently large to compensate for the derivative loss due to the lower regularity. We first consider the case $m=1$.
\begin{proof}[(\textbf{Proof of \eqref{3.5} for} {\boldmath$k=3,m=1$})] We will consider \textbf{Case (a)}-\textbf{Case (d)}.
	\begin{itemize}
		\begin{item} \textbf{Case (a)} {\boldmath$\left< L_1 \right> =\underset{j=1,2,3,4}{\max}\left< L_j \right>$.}\ First, we observe that,
			\be\label{3.34}
			\left< L_1 \right> ^{b^{\prime}}\sim\left< L_1 \right> ^{\frac{3p}{2}}\left< L_1 \right> ^{b^{\prime}-\frac{3p}{2}},\quad \prod_{j=1}^3{\left< L_j \right> ^{b^{\prime}}}\left< L_4 \right> ^b\ge \left< L_1 \right> ^{\frac{3p}{2}}\prod_{j=2,3,4}{\left< L_j \right> ^{b^{\prime}+\frac{b}{3}-\frac{p}{2}}}.
			\ee
			Moreover, the multiplier $M$ satisfies,
			\be\label{3.35}
			\frac{M}{\left< L_1 \right> ^{\frac{3p}{2}}}\les\frac{\prod_{j=1}^3{\left< \xi _j \right> ^{\frac{m-1}{2}}}}{\left< H \right> ^{\frac{3p}{2}}}\les\frac{\left< \xi _3 \right> ^{3p}}{\left< \xi _3 \right> ^{3p}}=1,
			\ee
			where we used $\left< \xi _3 \right> ^2\les\left< H_3 \right> \les \left< L_1 \right>$. Applying Hölder's inequality and the $L^6$ Strichartz estimate \eqref{2.30} yields,
			\be\label{3.36}
			\varGamma_{1}\les\prod_{j=2,3,4}{\lVert F_{\sigma_2}^{\left( j \right)} \rVert _{L_{t,x}^{6}}}\lVert F_{0}^{\left( 1 \right)} \rVert _{L_{t,x}^{2}}\les\prod_{j=1}^4{\lVert f_j \rVert _{L_{\tau,\xi}^{2}}},
			\ee
			provided that,
			\be\label{3.37}
			b^{\prime}>\frac{3p}{2},\qquad \sigma _2\doteq b^{\prime}+\frac{b}{3}-\frac{p}{2}>\frac{1}{2}.
			\ee
		\end{item}
	As in the proof for $k=0,1,2$, \textbf{Case (b)}-\textbf{Case (d)} can be proved similarly to \textbf{Case (a)}. We note that in \textbf{Case (d)}, the exponent transfer of $b$ requires the additional condition,
	\be\label{3.38}
	b>\frac{3p}{2}.
	\ee
	One can notice that \eqref{3.37}, \eqref{3.38} will force that $p<\fr{1}{3}$, due to $b,b'<\fr{1}{2}$. Recall that $p=-s$, thus, for $k=3$ and $m=1$, the estimate \eqref{3.5} holds as soon as,
	\be\label{3.39}
	b^{\prime},b>\frac{-3s}{2},\qquad  b^{\prime}+\frac{b}{3}>\frac{1-s}{2}.
	\ee
	\end{itemize}
This completes the proof.\end{proof}
\begin{proof}[(\textbf{Proof of \eqref{3.5} for} {\boldmath$k=3,\ m\ge2$})]
	For simplicity, we consider only the case $\left| \xi _3 \right|\ge 10, \left| \xi _2 \right|\ge 4$; the case $\left| \xi _3 \right|\le 10$ or $\left| \xi _2 \right|\le 4$ follows by the same arguments as in \textbf{Case A} and \textbf{Case B} of the proof for \(k=0,1,2\). We also consider \textbf{Cases (a)--(d)} below.
	\begin{itemize}
		\item \textbf{Case (a)} {\boldmath$\left< L_1 \right> =\underset{j=1,2,3,4}{\max}\left< L_j \right>$.}\ We begin with the observation that,
		\be\label{3.40}
		\langle L_1 \rangle^{b'} \sim \langle L_1 \rangle^{\frac{3p-m+1}{2m}}
		\langle L_1 \rangle^{b'-\frac{3p-m+1}{2m}},
		\ee
		and,
		\be\label{3.41}
		\prod_{j=1}^3 \langle L_j \rangle^{b'} \langle L_4 \rangle^b
		\ge \langle L_1 \rangle^{\frac{3p-m+1}{2m}}
		\prod_{j=2,3,4} \langle L_j \rangle^{b'+\frac{b}{3}-\frac{3p-m+1}{6m}}.
		\ee
		The multiplier $M$ satisfies the estimate,
		\be\label{3.42}
		\frac{M}{\langle L_1 \rangle^{\frac{3p-m+1}{2m}}}
		\lesssim \frac{\prod_{j=1}^3 \langle \xi_j \rangle^p}{\langle \xi_3 \rangle^{3p-m+1}}
		\lesssim \langle \xi_2 \rangle^{\frac{m-1}{2}} \langle \xi_3 \rangle^{\frac{m-1}{2}}
		\sim |\xi_2 \xi_3|^{\frac{m-1}{2}},
		\ee
		where we used the relations $\left| \xi_3 \right|^{2m} \lesssim \left| H_3 \right| \lesssim \left| L_1 \right|$ and $|\xi_3| \ge |\xi_2| \gtrsim 1$. Applying the Strichartz estimates \eqref{2.28} and \eqref{2.29} then gives,
		\be\label{3.43}
		\Gamma_1 \lesssim \|F_0^{(1)}\|_{L_{t,x}^2}
		\|D_x^{\frac{m-1}{2}} F_{\sigma_3}^{(2)}\|_{L_t^4 L_x^\infty}
		\|D_x^{\frac{m-1}{2}} F_{\sigma_3}^{(3)}\|_{L_t^4 L_x^\infty}
		\|F_{\sigma_3}^{(4)}\|_{L_t^\infty L_x^2}
		\lesssim \prod_{j=1}^4 \|f_j\|_{L_{\xi,\tau}^2},
		\ee
		provided that,
		\be\label{3.44}
		b' > \frac{3p-m+1}{2m}, \qquad
		\sigma_3 \doteq b' + \frac{b}{3} - \frac{3p-m+1}{6m} > \frac{1}{2}.
		\ee
		The remaining cases \textbf{Cases (b)--(d)} are similar to \textbf{Case (a)}, with the same modulation transfer argument. Two differences are worth noting: in \textbf{Case (c)}, $\left| \xi _4 \right|\sim\left| \xi _3 \right|$ or $\left| \xi _2 \right|\sim\left| \xi _3 \right|$, which permits an exponent transfer between the two factors; in \textbf{Case (d)}, the exponent transfer of \(b\) requires the following extra condition,
		\be\label{3.45}
		b > \frac{3p-m+1}{2m}. 
		\ee
     Since \(p = -s\), for \(k=3\) and \(m\ge2\), the estimate \eqref{3.5} holds whenever,
		\be\label{3.46}
		\frac{1-m-3s}{2m}<b', b<\frac{1}{2}, \qquad
		b' + \frac{b}{3} > \frac{1+2m-3s}{6m}.
		\ee
		This condition reduces to the one, \eqref{3.39}, obtained earlier when $m=1$.
	\end{itemize}
	
	The proof is now complete.
\end{proof}
\begin{rem}
	In the proof we have used the bound $\left< \xi _4 \right> ^p\gtrsim1$, thereby neglecting the influence of $\xi_4$. We note that a refined treatment of this term may yield sharper estimates, although this seems unlikely. In fact, compared with the result of Himonas and Yan\cite{himonas2024schrodinger}, this estimate may already be sharp. 
\end{rem}
\begin{rem}
	Due to \eqref{3.46}, $b', b<\frac{1}{2}$, we can not achieve $s=-\fr{2m-1}{3}$.
\end{rem}
\subsection{Proof of the temporal Bourgain spaces estimates \eqref{3.6}}
\hspace{1em}
To prove \eqref{3.6}, we consider two regions.
\be\label{3.47}
\varOmega _1\doteq\left\{ \left( \vec{\xi},\vec{\tau} \right) \in \mathbb{R}^4\times \mathbb{R}^4|\left| \tau _4 \right|\ge \frac{\left| \xi _4 \right|^{2m}}{4} \right\},\quad \varOmega _2\doteq\left\{ \left( \vec{\xi},\vec{\tau} \right) \in \mathbb{R}^4\times \mathbb{R}^4|\left| \tau _4 \right|\le \frac{\left| \xi _4 \right|^{2m}}{4} \right\}.
\ee
In the region $\varOmega _1$, the spatial Bourgain spaces estimates \eqref{3.5} is stronger than the temporal Bourgain spaces estimates \eqref{3.6}, while in the region $\varOmega _2$, it is weaker than \eqref{3.6}. More precisely, in $\varOmega _1$, we have $\left< \tau _4 \right> ^{\frac{-s}{2m}}\gtrsim\left< \xi _4 \right> ^{-s}$, so that $\varGamma_2\les\varGamma_1$, by \eqref{3.5} and \eqref{3.6}, hence, \eqref{3.6} follows as a corollary of \eqref{3.5}. In $\varOmega _2$, thanks to the condition $\left| \tau _4 \right|\le \frac{\left| \xi _4 \right|^{2m}}{4}$, we have the additional relation $\left< L_4 \right> \sim\left< \xi _4 \right> ^{2m}$. We therefore focus solely on $\varOmega _2$.
\begin{proof}[(\textbf{Proof of \eqref{3.6} for} {\boldmath$k=0,1,2$} \textbf{and} {\boldmath$m\ge2$})]
 The case for $k=0,1,2$ and $m=1$ follows from the argument in Remark \ref{rem 3.2}, we therefore treat $m\ge2$ below. As before, we consider \textbf{Case A} to \textbf{Case E}. 
 Note that for any $s\ge s_{m,k}$, since $\left| \tau _4 \right|\le \frac{\left| \xi _4 \right|^{2m}}{4}$, we have,
 $$\frac{\prod_{j=1}^3{\left< \xi _j \right> ^{-s}}}{\left< \tau _4 \right> ^{\frac{-s}{2m}}}=\frac{\prod_{j=1}^3{\left< \xi _j \right> ^{-s_{m,k}}}}{\left< \tau _4 \right> ^{\frac{-s_{m,k}}{2m}}}\left( \frac{\left< \tau _4 \right>}{\prod_{j=1}^3{\left< \xi _j \right> ^{2m}}} \right) ^{\frac{s-s_{m,k}}{2m}}\les\frac{\prod_{j=1}^3{\left< \xi _j \right> ^{-s_{m,k}}}}{\left< \tau _4 \right> ^{\frac{-s_{m,k}}{2m}}}\left( \frac{\left< \xi _4 \right> ^{2m}}{\prod_{j=1}^3{\left< \xi _j \right> ^{2m}}} \right) ^{\frac{s-s_{m,k}}{2m}}.$$ The factor in parenthesis is bounded; hence, we only need to treat $s=s_{m,k}$. Nevertheless, in \textbf{Case A}, \textbf{Case C} and \textbf{Case E}, we have used the bound $\left< \xi _4 \right> ^{-\frac{m-1}{2}}\les1$, so that the same arguments as above apply. Thus it remains to treat only \textbf{Case B} and \textbf{Case D} in $\varOmega _2$. Following the same pattern as above, we consider the subcases \textbf{(a)} to \textbf{(d)} separately. Recall that $K_4$ was decomposed into the regions $A$, $B$, $C$, $D$. We now define the corresponding localized integrals and a new multiplier,
\be\label{3.48}
\varGamma _{2,\varOmega _2,B}\doteq\left| \int_{K_4\cap \varOmega _2\cap B}{\frac{\prod_{j=1}^3{\left< \xi _j \right> ^{\frac{m-1}{2}}}\prod_{j=1}^4{f_j\left( \xi _j,\tau _j \right)}}{\left< \tau _4 \right> ^{\frac{m-1}{4m}}\left< L_4 \right> ^{b}\prod_{j=1}^3{\left< L_j \right> ^{b^{\prime}}}}} \right|,
\ee
\be\label{3.49}
\varGamma _{2,\varOmega _2,D}\doteq\left| \int_{K_4\cap \varOmega _2\cap D}{\frac{\prod_{j=1}^3{\left< \xi _j \right> ^{\frac{m-1}{2}}}\prod_{j=1}^4{f_j\left( \xi _j,\tau _j \right)}}{\left< \tau _4 \right> ^{\frac{m-1}{4m}}\left< L_4 \right> ^{b}\prod_{j=1}^3{\left< L_j \right> ^{b^{\prime}}}}} \right|,
\ee
\be\label{3.50}
M\doteq\frac{\prod_{j=1}^3{\left< \xi _j \right> ^{\frac{m-1}{2}}}}{\left< \tau _4 \right> ^{\frac{m-1}{4m}}}\les\prod_{j=1}^3{\left< \xi _j \right> ^{\frac{m-1}{2}}}.
\ee
\begin{itemize}
	\item \textbf{Case B.} {\boldmath $\left| \xi _3 \right|\ge 10, \left| \xi _2 \right|\le 4$}\quad We take advantage of the relation $\left< L_4 \right> \sim\left< \xi _4 \right> ^{2m}$ by writing,
	\be\label{3.51}
	\left< L_4 \right> ^b\sim\left< L_4 \right> ^{b-\frac{m-1}{4m}}\left< \xi _4 \right> ^{2m\cdot \frac{m-1}{4m}},
	\ee
	so that the bound, 
	\be\label{3.52}
	\frac{M}{\left< \xi _4 \right> ^{\frac{m-1}{2}}}\les\frac{\left< \xi _3 \right> ^{\frac{m-1}{2}}}{\left< \xi _4 \right> ^{\frac{m-1}{2}}}\sim1,
	\ee
	holds, where we used $\left| \xi _4 \right|\sim\left| \xi _3 \right|$. Then the estimate follows as in \eqref{3.10} with the additional condition,
	\be\label{3.53}
	b-\frac{m-1}{4m}>\fr{2m+1}{8m}.
	\ee
	\item\textbf{Case D.} {\boldmath $\left| \xi _3 \right|\ge 10, \left| \xi _2 \right|\ge 4, \left| \xi _1 \right|\ge 1, \left| \xi _4 \right|\ge \frac{\left| \xi _1 \right|}{5}$,}
	\subitem \textbf{Case D.(a)} {\boldmath$\left< L_1 \right> =\underset{j=1,2,3,4}{\max}\left< L_j \right>$}\quad We first observe that,
	\be\label{3.54}
	\left< L_4 \right> ^{b}\sim\left< \xi _4 \right> ^{2m\cdot \frac{m-1}{4m}}\left< L_4 \right> ^{b -\frac{m-1}{4m}},\quad \frac{M}{\left< \xi _4 \right> ^{\frac{m-1}{2}}}\les\left< \xi _2 \right> ^{\fr{m-1}{2}}\left< \xi _3 \right> ^{\fr{m-1}{2}}\les\left< \xi _3 \right> ^{m-1}\les\frac{\left| \xi _3 \right|^{m-\frac{1}{2}}}{\left| \xi _2\xi _4 \right|^{\frac{1}{4}}}.
	\ee
	Moreover, by a modulation exponent transfer, we have,
	\be\label{3.55}
	\left< L_4 \right> ^{b -\frac{m-1}{4m}}\prod_{j=1}^3{\left< L_j \right> ^{b^{\prime}}}\gtrsim\prod_{j=2,3,4}{\left< L_j \right> ^{b^{\prime}+\frac{b}{3}-\frac{m-1}{12m}}}.
	\ee
	 Combining these estimates yields,
	\be\label{3.56}
	\varGamma _{2,\varOmega _2,D}\les\lVert F_{0}^{\left( 1 \right)} \rVert _{L_{x,t}^{2}}\lVert D_{x}^{-\frac{1}{4}}F_{\sigma_1}^{\left( 2 \right)} \rVert _{L_{x}^{4}L_{t}^{\infty}}\lVert D_{x}^{m-\frac{1}{2}}F_{\sigma_1}^{\left( 3 \right)} \rVert _{L_{x}^{\infty}L_{t}^{2}}\lVert D_{x}^{-\frac{1}{4}}F_{\sigma_1}^{\left( 4 \right)} \rVert _{L_{x}^{4}L_{t}^{\infty}}\les\prod_{j=1}^4{\lVert f_j \rVert _{L_{\xi ,\tau}^{2}}},
	\ee
	provided that,
	\be\label{3.57}
 \sigma_1\doteq b^{\prime}+\frac{b}{3}-\frac{m-1}{12m} >\fr{1}{2}.
	\ee
	The remaining subcases \textbf{ D.(b)}, \textbf{ D.(c)}, and \textbf{ D.(d)} can be proved by similar arguments; we omit the details. We point out, however, that in \textbf{Case D.(c)}, the relation $\left| \xi _3 \right|\sim\left| \xi _2 \right|$ or $\left| \xi _3 \right|\sim\left| \xi _4 \right|$ is needed to transfer the modulation exponents.\\ We therefore conclude that the temporal Bourgain space estimate \eqref{3.6} is valid under the conditions \eqref{3.31}, \eqref{3.53} and \eqref{3.57}. More precisely, for $k=0,1,2$, the spatial Bourgain spaces estimate \eqref{3.5} and the temporal Bourgain space estimate \eqref{3.6} hold for any $b$, $b^{\prime}$ satisfying
	\be\label{3.58}
	\frac{7}{18}+\frac{1}{36m}<b^{\prime}<\fr{1}{2},\quad \fr{1}{2}-\fr{1}{8m}<b<\fr{1}{2},\quad b^{\prime}+\fr{b}{3}>\fr{7}{12}-\fr{1}{12m}.
	\ee\end{itemize}  This completes the proof. \end{proof}
\begin{proof}[(\textbf{Proof of \eqref{3.6} for} {\boldmath$k=3$})] Finally, we note that the bound $\left< \xi _4 \right> ^{-s}\les1$ is used throughout the proof of \eqref{3.5} for $k=3$. The proof of \eqref{3.6} follows by the same argument since $\left< \tau _4 \right> ^{\frac{-s}{2m}}\gtrsim1$, so we do not repeat it. Now we can conclude that for $k=3$, the spatial Bourgain spaces estimate \eqref{3.5} and the temporal Bourgain space estimate \eqref{3.6} hold for any $b'<b$ satisfying
	\be\label{3.59}
	\frac{1-m-3s}{2m}<b'<b<\frac12,
	\qquad
	b'+\frac{b}{3}>\frac{2m+1-3s}{6m}.
	\ee
	The proof is now complete.\end{proof}
\section{Proof of well-posedness on the half line, Theorem \ref{thm 1.5}}\label{sec 4}
\hspace{1em}
Using the Fokas formula, the known linear estimates(see Lemma \ref{lem 4.1}), and our new trilinear estimates, the standard contraction mapping argument gives local well-posedness of the IBVP. A proof is included for completeness. First, we recall some existing results from the literature, including the solution of the linear IVP and IBVP, as well as the corresponding linear estimates, which play a key role in the proof of well-posedness. We begin by introducing the linear IBVP \eqref{1.5} with forcing $f$.
\be\label{4.1}
\left\{ \begin{array}{ll}
	iu_t+\left( -1 \right) ^{m+1}\partial _{x}^{2m}u=f\left(x,t\right),\ x\in\mathbb{R}^{+},\ t\in\left(0,T\right),\\
	u\left( x,0 \right) =u_0\left( x \right) ,\\
	u\left( 0,t \right) =g_0\left( t \right) ,\ \partial _{x}^{j}u\left( 0,t \right) =g_j\left( t \right) ,\ j=1,2,...,m-1.\\
\end{array} \right.
\ee
Next, we present the linear estimate with the following lemma established in \cite{himonas2024higher}.
\begin{lem}\cite{himonas2024higher,himonas2026higher}\label{lem 4.1}
	Suppose that
	$-m - \frac{1}{2} < s \le0$, $0 < b < \frac{1}{2}$.
	Then, \eqref{4.1} admits a solution 
	\(u \in X_{\mathbb{R}^{+},T}^{s,b}\), which satisfies the estimate
	\be\label{4.2}
	\begin{aligned}
		&\| W[u_0,g_0,\dots,g_{j-1};f] \|_{X_{\mathbb{R}^{+},T}^{s,b}} \\
		\les& \Big[
		\| u_0 \|_{H_x^s(\mathbb{R}^{+})} 
		+ \sum_{j=0}^{m-1} \| g_\ell \|_{H_t^{\frac{2s+2m-1-2j}{4m}}(0,T)}
		+ \| f \|_{X_{\mathbb{R}^{+},T}^{s,-b}}
		+ \| f \|_{Y_{\mathbb{R}^{+},T}^{s,-b}}
		\Big],
	\end{aligned}
	\ee
	where \(Y^{s,b}\) are the temporal Bourgain spaces defined by the norm in \eqref{1.11}.
\end{lem}
To close the contraction mapping argument, we need a multiplier estimate.
\begin{lem}\cite{himonas2024higher,himonas2026higher}\label{lem 4.2}
	Let $\eta(t)$ be a function in the Schwartz space $\mathcal{S}(\mathbb{R})$. 
	If $-\frac{1}{2} < b' \le b < \frac{1}{2}$,
	then for any $0 < T^* \le 1$ we have
	\be\label{4.3}
	\bigl\| \eta(\fr{t}{T^*}) \, u \bigr\|_{X^{s,b'}}
	\les {T^*}^{b - b'} \,
	\| u \|_{X^{s,b}}.
	\ee
\end{lem}
	\begin{proof}[(\textbf{Proof of Theorem \ref{thm 1.5}})]
		Let $m\ge 1$ be an integer and let $s$ satisfy $\eqref{1.28}$. Choose $b,b'$ as in Theorem \ref{thm 1.4} and set
		$\delta\doteq b-b'>0$; by Remark \ref{rem 1.5}, for $k=0,1,2$ the gap $\delta$ can be
		fixed as a constant independent of the data, while for $k=3$ one has
		$0<\delta<\frac{2m-1+3s}{6m}$. We assume $0<T\le 1$. Let $c_1,c_2,c_3$ denote the
		constants in Lemma \ref{lem 4.1}, Theorem \ref{thm 1.4} and Lemma \ref{lem 4.2}, respectively, and
		set, as in \eqref{1.29},
		\be\label{4.4}
		R\doteq\|u_0\|_{H^{s}(\mathbb{R}^+)}
		+\sum_{j=0}^{m-1}\|g_j\|_{H^{\frac{2s+2m-1-2j}{4m}}(0,T)}.
		\ee
		For $0<T^*\le T$, we define the iteration map $\Phi_{T^*}$ by replacing the forcing $f$
		in \eqref{4.1} with the nonlinearity localized by
		$\psi_{2T^*}(t)=\psi(\tfrac{t}{2T^*})$, where $\psi\in C_c^\infty(-1,1)$ and
		$\psi\equiv 1$ on $[-\tfrac12,\tfrac12]$:
		\be\label{4.5}
		\Phi_{T^*}u\doteq W\big[u_0,g_0,\dots,g_{m-1};\ \psi_{2T^*}(t)N_k(u,u,u)\big],
		\qquad (x,t)\in\mathbb{R}^+\times(0,T).
		\ee
		We will show that $\Phi_{T^*}$ is a contraction on the closed ball,
		\be\label{4.6}
		B(r):=\Big\{u\in X^{s,b}_{\mathbb{R}^+,T}:\ \|u\|_{X^{s,b}_{\mathbb{R}^+,T}}\le r\Big\},
		\qquad r\doteq2c_1R ,
		\ee
		of the Banach space $X^{s,b}_{\mathbb{R}^+,T}$.\\
		\medskip
		\noindent\textbf{Step 1 ($\Phi_{T^*}$ maps $B(r)$ into itself).}
		By the assumptions on $s$, the linear estimate \eqref{4.2} gives,
		\be\label{4.7}
		\|\Phi_{T^*}u\|_{X^{s,b}_{\mathbb{R}^+,T}}
		\le c_1\Big(R+\|\psi_{2T^*}N_k(u,u,u)\|_{X^{s,-b}_{\mathbb{R}^+,T}}
		+\|\psi_{2T^*}N_k(u,u,u)\|_{Y^{s,-b}_{\mathbb{R}^+,T}}\Big).
		\ee
		Choose an extension $\tilde u\in X^{s,b}(\mathbb{R}^2)$ of $u$ satisfying
		$\|\tilde u\|_{X^{s,b}}\le 2\|u\|_{X^{s,b}_{\mathbb{R}^+,T}}$. Since $\psi_{2T^*}$ is
		real-valued, then,
		\be\label{4.8}
		\psi_{2T^*}N_k(u,u,u)\big|_{\mathbb{R}^+\times(0,T)}
		=N_k(\psi_{2T^*}\tilde u,\tilde u,\tilde u)\big|_{\mathbb{R}^+\times(0,T)} .
		\ee
		Hence, by the definition of the restriction norm and Theorem \ref{thm 1.4}, one
		has,
		\be\label{4.9}
		\|\psi_{2T^*}N_k(u,u,u)\|_{X^{s,-b}_{\mathbb{R}^+,T}}
		\le \|N_k(\psi_{2T^*}\tilde u,\tilde u,\tilde u)\|_{X^{s,-b}}
		\le c_2\|\psi_{2T^*}\tilde u\|_{X^{s,b'}}\|\tilde u\|^2_{X^{s,b'}} ,
		\ee
		and identically in $Y^{s,-b}$. Applying the multiplier estimate \eqref{4.3} with
		$b'<b$, we obtain,
		\be\label{4.10}
		\|\psi_{2T^*}\tilde u\|_{X^{s,b'}}
		\le c_3(T^*)^{b-b'}\|\tilde u\|_{X^{s,b}}
		=c_3(T^*)^{\delta}\|\tilde u\|_{X^{s,b}} ,
		\ee
		and since $\|\tilde u\|_{X^{s,b'}}\le\|\tilde u\|_{X^{s,b}}\le 2r$, we can conclude
		that,
		\be\label{4.11}
		\|\psi_{2T^*}N_k(u,u,u)\|_{X^{s,-b}_{\mathbb{R}^+,T}}
		+\|\psi_{2T^*}N_k(u,u,u)\|_{Y^{s,-b}_{\mathbb{R}^+,T}}
		\le 16c_2c_3(T^*)^{\delta}r^3 .
		\ee
		Substituting \eqref{4.11} into \eqref{4.7} yields,
		\be\label{4.12}
		\|\Phi_{T^*}u\|_{X^{s,b}_{\mathbb{R}^+,T}}
		\le c_1R+16c_1c_2c_3(T^*)^{\delta}r^3 .
		\ee
		Thus, if we choose $r\doteq2c_1R$, then $\Phi_{T^*}(B(r))\subset B(r)$ as soon as
		\be\label{4.13}
		16c_1c_2c_3(T^*)^{\delta}r^2\le\tfrac12 .
		\ee
		\medskip
		\noindent\textbf{Step 2 ($\Phi_{T^*}$ is a contraction).}
		Let $u,v\in B(r)$. Choose an extension $\tilde v$ of $v$ with
		$\|\tilde v\|_{X^{s,b}}\le 2r$, an extension $w$ of $u-v$ with
		$\|w\|_{X^{s,b}}\le 2\|u-v\|_{X^{s,b}_{\mathbb{R}^+,T}}$, and set
		$\tilde u\doteq\tilde v+w$. Then $\tilde u$ extends $u$ with,
		\be\label{4.14}
		\|\tilde u\|_{X^{s,b}}\le 6r,
		\qquad
		\|\tilde u-\tilde v\|_{X^{s,b}}\le 2\|u-v\|_{X^{s,b}_{\mathbb{R}^+,T}} .
		\ee
		Since $N_k$ is trilinear (each entry being either linear or conjugate-linear) and
		$\psi_{2T^*}$ is real-valued, one has
		\be\label{4.15}
			\psi_{2T^*}\big[N_k(\tilde u,\tilde u,\tilde u)-N_k(\tilde v,\tilde v,\tilde v)\big]
			=N_k\big(\psi_{2T^*}(\tilde u-\tilde v),\tilde u,\tilde u\big)
			+N_k\big(\psi_{2T^*}\tilde v,\tilde u-\tilde v,\tilde u\big)
			+N_k\big(\psi_{2T^*}\tilde v,\tilde v,\tilde u-\tilde v\big).
		\ee
		By \eqref{4.15}, Theorem \ref{thm 1.4} and the multiplier estimate \eqref{4.3},
		and using $b'<b$, we obtain,
		\be\label{4.16}
		\Big\|\psi_{2T^*}\big[N_k(\tilde u,\tilde u,\tilde u)-N_k(\tilde v,\tilde v,\tilde v)\big]
		\Big\|_{X^{s,-b}}
		\le c_2c_3(T^*)^{\delta}\|\tilde u-\tilde v\|_{X^{s,b}}
		\Big(\|\tilde u\|_{X^{s,b}}^2+\|\tilde u\|_{X^{s,b}}\|\tilde v\|_{X^{s,b}}
		+\|\tilde v\|_{X^{s,b}}^2\Big),
		\ee
		and the same bound holds in $Y^{s,-b}$. Therefore, by the linear estimate \eqref{4.7},
		it follows that
		\be\label{4.17}
		\begin{aligned}
			\|\Phi_{T^*}u-\Phi_{T^*}v\|_{X^{s,b}_{\mathbb{R}^+,T}}
			&\le 2c_1c_2c_3(T^*)^{\delta}\|\tilde u-\tilde v\|_{X^{s,b}}
			\Big(\|\tilde u\|_{X^{s,b}}^2+\|\tilde u\|_{X^{s,b}}\|\tilde v\|_{X^{s,b}}
			+\|\tilde v\|_{X^{s,b}}^2\Big)\\
			&\le 2c_1c_2c_3(T^*)^{\delta}\cdot 2\|u-v\|_{X^{s,b}_{\mathbb{R}^+,T}}
			\cdot(36+12+4)\,r^2\\
			&=208c_1c_2c_3(T^*)^{\delta}r^2\,\|u-v\|_{X^{s,b}_{\mathbb{R}^+,T}} .
		\end{aligned}
		\ee
		Hence $\Phi_{T^*}$ is a contraction on $B(r)$ provided that
		\be\label{4.18}
		208c_1c_2c_3(T^*)^{\delta}r^2\le\tfrac12 ,
		\ee
		which clearly implies the onto condition \eqref{4.13}.\\
		\medskip
		\noindent\textbf{Step 3 (Lifespan).}
		Solving \eqref{4.18} for $T^*$ and recalling that $r=2c_1R$ gives
		\be\label{4.19}
		T^*\le\big(416c_1c_2c_3\,r^2\big)^{-\fr{1}{\delta}}
		=\big(1664c_1^2c_2c_3\,R^2\big)^{-\fr{1}{\delta}}.
		\ee
		Since $\langle R\rangle=1+R\ge\max(R,1)$, we can choose
		\be\label{4.20}
		T_0\doteq c_0\langle R\rangle^{-\fr{2}{\delta}},
		\qquad
		c_0\doteq\min\Big\{T,\ \big(1664c_1^2c_2c_3\big)^{-\fr{1}{\delta}}\Big\},
		\ee
		so that $0<T_0<T$ and \eqref{4.18} holds with $T^*=T_0$. Thus,
		\be\label{4.21}
		T_0\sim\langle R\rangle^{-\fr{2}{\delta}},
		\ee
		which is the second half of \eqref{1.30}. Note that, in contrast to the lifespan
		$T_0\sim\langle R\rangle^{-\fr{4}{\beta_2}}$ of Himonas and Yan, where
		$\beta_2=\beta_2(s,m)\to 0$ as $s\to s_m$, for $k=0,1,2$ the exponent $\fr{2}{\delta}$ here
		remains bounded at $s=s_m$ (see Remark \ref{rem 1.5}); for $k=3$ the gap
		$\delta$ vanishes only as $s\to s^*_m$.\\
		\medskip
		\noindent\textbf{Step 4 (Conclusion).}
		Since $B(r)$ is closed in the Banach space $X^{s,b}_{\mathbb{R}^+,T}$, the contraction
		mapping principle yields a unique $u\in B(r)$ with $\Phi_{T_0}u=u$. As
		$\psi_{2T_0}(t)=1$ for $|t|\le T_0$, the restriction of $u$ to
		$\mathbb{R}^+\times(0,T_0)$ solves the IBVP \eqref{1.3} on $(0,T_0)$, and by
		\eqref{4.12} it satisfies
		$\|u\|_{X^{s,b}_{\mathbb{R}^+,T_0}}\le r=2c_1R$, which is the first half of
		\eqref{1.30}. The Lipschitz continuity of the data-to-solution map follows directly
		from the contraction mapping principle. The proof is now complete.
	\end{proof}

\section*{Acknowledgements}
S.  Li is supported by the National Natural Science Foundation of China (No.  12001084 )

\end{document}